\documentclass[10pt,a4paper]{article}

\usepackage{isorot}

\usepackage{graphicx,subfigure}

\usepackage[square,numbers]{natbib}

\usepackage{booktabs}
\usepackage{multirow}
\usepackage{hyperref}

\usepackage{pstricks,pst-plot}

\usepackage{amsmath}
\usepackage{amsthm}
\usepackage{amsfonts}
\usepackage{amssymb}
\usepackage{latexsym}
\usepackage{stackrel}
\usepackage[T1]{fontenc}
\usepackage{enumerate}
\usepackage{fourier}

\usepackage{lastpage}

\let\mathnumsetfont\mathbb
\newcommand\Rset{\mathnumsetfont R} 

\newtheorem{teo}{Theorem}
\newtheorem{teo*}{Theorem}

\newtheorem{rmq}{Remark}[section]

\newtheorem{lem}{Lemma}[section]

\newcommand\M{{\cal M}}

\def\limsup{\mathop{\overline{\mathrm{lim}}}}
\def\liminf{\mathop{\underline{\mathrm{lim}}}}
\newcommand\limx{\lim_{x\rightarrow\infty}}
\newcommand\limn{\lim_{n\rightarrow\infty}}
\newcommand\limsupx{\limsup_{x\rightarrow\infty}}
\newcommand\liminfx{\liminf_{x\rightarrow\infty}}

\def\barF{\overline{F}}

\title{A simple estimator for the $\M$-index of functions in $\M$}

\author{Meitner Cadena\thanks{UPMC Paris 6 \& CREAR, ESSEC Business School;\, E-mail: meitner.cadena@etu.upmc.fr or b00454799@essec.edu or meitner.cadena@gmail.com}} 

\begin{document}

\maketitle

\begin{abstract}
An estimator for the $\M$-index of functions of $\M$, a larger class than the class of
regularly varying (RV) functions, is proposed. This index is the tail index of RV functions
and this estimator is thus a new one on the class of RV functions. This estimator
satisfies, assuming suitable conditions, weak and strong consistencies. Asymptotic normality of
this estimator is proved for a large class of RV functions, showing a better performance than some well-known estimators.
Illustrations with simulated and real life datasets are provided.

\vspace{5mm}


\vspace{1mm}

\end{abstract}

\vspace{5mm}

\section{Introduction}

A main concern in extreme value theory, involved in many fields of application, such as e.g. hydrology, biology and finance, is the examination of the tail heaviness of a distribution.
This analysis is deeply concentrated on random variables (rv) having survival functions $\barF$ being regularly varying (RV) functions, i.e. for some $\alpha\geq0$, called the tail index of $\barF$,
\begin{equation}\label{eq:20141020:001}
\limx\frac{\barF(tx)}{\barF(x)}=t^{-\alpha}\textrm{\quad($t>0$).}
\end{equation}
%
The class of RV functions with tail index $\alpha$ is denoted by $RV_{-\alpha}$. 

(\ref{eq:20141020:001}) shows that the notion of tail index is intimately related to the heaviness of the tail of a distribution.
For instance, lower values of tail indices correspond to heavier tails.
Hence, a vast literature has been focused on the estimation of this index and its applications as the estimation of extreme quantiles.
See e.g. \cite{BeirlantGoegebeurTeugelsSegersDeWaalFerro} and \cite{deHaanFerreira} and the references
therein for discussions on estimations of the tail index.
A number of estimators of the tail index from a random sample $X_1$, $X_2$, \ldots, $X_n$ have been proposed.
%
The most well-known estimator is one proposed by Hill \cite{Hill1975} in 1975, which shows favorable features to be used than other competitors (see e.g. \cite{Smith1987}, pp. 1181-1182, and \cite{Segers2001}) in terms of rate of convergence.

More recently, Cadena and Kratz introduced the class $\M$ that consists of measurable functions $U:\Rset^+\to\Rset^+$ satisfying (see e.g. \cite{NN2014}, \cite{NN2015CKanalysis} and \cite{NN2015CKbis})
\begin{equation}\label{MkappaBis}
\exists\alpha\in\Rset, \forall \varepsilon > 0, \lim_{x\rightarrow\infty}\frac{U(x)}{x^{\alpha+\epsilon}}=0\text{\quad and\quad}\lim_{x\rightarrow\infty}\frac{U(x)}{x^{\alpha-\epsilon}}=\infty\textrm{.}
\end{equation}
It is not hard to see that the parameter $\alpha$ in \eqref{MkappaBis} is unique, hence called the $\M$-index of $U$, and that $\epsilon$ may be taken sufficiently small.
%
$\M$ is strictly larger than the class of RV functions and the $\M$-index of $U\in\M$ is its tail index if $U$ is RV (see e.g. Proposition 2.1 in \cite{NN2014} or Theorem 1.1 in \cite{NN2015CKanalysis}).

These authors showed several ways to characterize $\M$, i.e. \eqref{MkappaBis}.
The following characterization of $\M$ given by Theorem 1.1 in \cite{NN2014} or Theorem 2.7 in \cite{NN2015CKanalysis} is of interest in this manuscript.
For the sake of completeness of this manuscript, its proof is given in appendix, which is copied from Theorem 1.1 in \cite{NN2014}.

\begin{teo}[Cadena and Kratz, Theorem 1.1 in \cite{NN2014} or Theorem 2.7 in \cite{NN2015CKanalysis}]\label{teoCK}
Let $\eta\in\Rset$ and let $U : \Rset^+ \to \Rset^+$ be a measurable function. Then, $U\in\M$ with $\M$-index $\eta$, if and only if
\begin{equation}\label{teo1.1}
\limx\frac{\log\,U(x)}{\log(x)}=\eta\textrm{.}
\end{equation}
\end{teo}

Inspired by \eqref{teo1.1}, which is related to quantile-quantile plots incorporating logarithm transformations that are frequently used in the extreme value analysis (see e.g. \cite{BeirlantVynckierTeugels1996}) and have motivated the formulation of estimators of the tail index (see e.g. \cite{KratzResnick1996}), we propose some estimators of the $\M$-index of a tail of distribution belonging to $\M$.
Since any RV function belongs to $\M$, these estimators are new ones for the tail index.
Differences among these estimators and some well-known estimators are pointed out, namely on convergence rates.

In Section \ref{Mainresults} we define an estimator for the $\M$-index of $U\in\M$ and give some of its properties, namely its weak and strong consistency, and its asymptotic normality supposed the tail being RV.
In Section \ref{applications} numerical illustrations are provided on simulated and real data, comparing the performance of our new estimator with some existing ones, when the shape and scale of a tail of distribution vary as well the sample size.
The last section presents conclusions.

\section{Main results}
\label{Mainresults}

All over this paper $X$ is a positive random variable (r.v.) with distribution $F$ and distribution tail denoted by $\barF=1-F$.
Suppose that the endpoint of $F$ is infinite, i.e. $x^*=\sup\big\{x:F(x)<1\big\}=\infty$ (if $x^*<\infty$, applying the change of variable $y=1\big/(x^*-x)$ one obtains a positive variable with infinite endpoint).

Let $X_1$, \ldots, $X_n$ be independent and identically distributed (i.i.d.) r.v.s with distribution $F$ and let $X_{1:n}\leq\cdots\leq X_{n:n}$ be the order statistics of this sample.
The empirical distribution function of $F$ is defined by, for $x>0$, 
$$
F_n(x)=\frac{1}{n}\sum_{i=1}^n1_{\left\{X_i\leq x\right\}}\textrm{.}
$$
We denote by $\barF_n=1-F_n$ its empirical tail of distribution.


\subsection{Defining a 
simple estimator for the $\M$-index}

Assume $\barF\in\M$ with $\M$-index $-\alpha$ ($<0$).
Notice that 
if $\barF$ is regularly varying, $\barF\in\textrm{RV}_{-\alpha}$.

Theorem 1.1 in \cite{NN2014} claims that $\barF$ satisfies \eqref{teo1.1}.
Since $\barF_n(X_{n-k:n})=k\big/n$ ($0\leq k<n$),
%
%
this property is related with the scatterplot with coordinates
$$
\left\{\left(\log\left(n\big/k\right),\log(X_{n-k:n})\right),1\leq k< n\right\}.
$$
known as the Pareto quantile plot (see e.g. \cite{BeirlantVynckierTeugels1996}).
This plot is often used to analyze data behaviors related to extreme values.
For instance, when the sample satisfy a strict Pareto behavior this relationship is linear.
Also, this type of coordinates inspired the formulation of estimators for the tail index, for instance the called qq-estimator (see \cite{KratzResnick1996}).

%

From \eqref{teo1.1} a natural estimator for $\alpha$ may be formulated.
This is our proposal of estimator, which is given by, for $1\leq k<n$,
$$
\hat{\alpha}=-\frac{\log\,\barF_n(X_{n-k:n})}{\log(X_{n-k:n})}\textrm{,}
$$
that may be rewritten as
\begin{equation}\label{eq:20150316:010}
\hat{\alpha}=\frac{\log\left(n\big/k\right)}{\log(X_{n-k:n})}\textrm{.}
\end{equation}
We see that this estimator is also based on the largest values of the sample but unlike most of well-known estimators a unique order statistic is used.
This does not consider necessarily the top values, which are involved in the well-known estimators.
Also, the heaviness of tails of some distributions, other than RV distributions, could be analyzed,
for instance $\barF(x)=e^{-\lfloor \log(x)\rfloor}$ in a neighborhood of $\infty$, where $\lfloor x\rfloor$ means the integer lower than or equal to $x$.

Convergence in distribution, convergence in probability and almost sure convergence are denoted by $\overset{d}{\to}$, $\overset{P}{\to}$ and $\overset{a.s.}{\to}$ respectively.

\subsection{Consistency of the new estimator}

\subsubsection{Weak consistency}
\label{consistency}

We consider sequences of positive integers $k=k(n)$ satisfying
\begin{equation}\label{kn}
\textrm{$1\leq k\leq n-1$,\quad $k\to\infty$, \quad and \quad $k\big/n\to0$ \quad as \quad $n\to\infty$.}\tag{\textrm{k}}
\end{equation}

In order to prove weak consistency of $\hat{\alpha}$ we need the following standard results:


\begin{lem}\label{lem:main:lem3.2.1}
Let $X_1$, $X_2$, \ldots be i.i.d. r.v.s with distribution $F:\Rset^+\to\Rset^+$ such that $F<1$.
Let $X_{1:n}\leq\cdots\leq X_{n:n}$ be its order statistics.
If $k=k(n)$ satisfies \eqref{kn}, then we have
$$
X_{n-k:n}\overset{a.s.}{\to}\infty\quad\textrm{as\quad $n\rightarrow\infty$.}
$$
\end{lem}

\begin{proof}
We proof this lemma inspired by Lemma 3.2.1 given in \cite{deHaanFerreira}.
%
%

We proceed by contradiction.
Suppose that there exists $r$ such that $X_{n-k:n} < r$ infinitely often and $\barF(r)>0$.
Then we would have, for large $n$,


$$
\frac{1}{n}\sum_{i=1}^n1_{\{X_i>X_{n-k:n}\}}
\ =\ \frac{k}{n}
\ \geq\ \frac{1}{n}\sum_{i=1}^n1_{\{X_i>r\}}
\ >\ \frac{\overline{F}(r)}{2}
$$

infinitely often, which is a contradiction when $n\rightarrow\infty$, because $k\big/n\rightarrow0$.
\end{proof}

\begin{lem}\label{teo:20150227:001}
Let $E_1$, \ldots, $E_n$ be i.i.d. r.v.s with standard exponential distribution $F$.
Let $E_{1:n}\leq\cdots\leq E_{n:n}$ be its order statistics.
If \eqref{kn} is satisfied, then we have
$$
\sqrt{k}\left(E_{n-k:n}-\log\left(\frac{n}{k}\right)\right)
$$
is asymptotically standard normal as $n\to\infty$.
\end{lem}

\begin{proof}
%
%
Let $U=(1-F)^\leftarrow$, $U(x)=\inf\big\{y:F(y)\geq1-1\big/y\big\}$.
Since $1-F(x)=e^{-x}$ we have $U(x)=\log(x)$.
%
%
%
%
Then, by Theorem 2.1.1 of \cite{deHaanFerreira},
$$
\sqrt{k}\frac{E_{n-k:n}-U\left(\frac{n}{k}\right)}{\frac{n}{k}U'\left(\frac{n}{k}\right)} \ = \ \sqrt{k}\frac{E_{n-k:n}-\log\left(\frac{n}{k}\right)}{\frac{n}{k}\times\frac{k}{n}} 
	\ = \ \sqrt{k}\left(E_{n-k:n}-\log\left(\frac{n}{k}\right)\right)
$$
is asymptotically standard normal as $n\to\infty$, supposing \eqref{kn}.
\end{proof}

\begin{lem}\label{lem:20150324:001}
Let $E_1$, \ldots, $E_n$ be i.i.d. r.v.s with standard exponential distribution $F$.
Let $E_{1:n}\leq\cdots\leq E_{n:n}$ be its order statistics.
If \eqref{kn} is satisfied, then we have
$$
\frac{E_{n-k:n}}{\log\left(\frac{n}{k}\right)}\overset{P}{\to}1\quad\textrm{as\quad $n\to\infty$.}
$$
\end{lem}

\begin{proof}
Application of Lemma \ref{teo:20150227:001} and e.g. \cite{vanderVaart}, exercice 18, page 24, or \cite{BishopFienbergHolland2007}, exercice 6, page 484.
\end{proof}

We can prove:

\begin{teo}[{Weak consistency of $\hat{\alpha}$}]\label{prop:20150324:001}

Let $X$, $X_1$, \ldots, $X_n$ be i.i.d. r.v.s with continuous distribution $F$ such that $\barF\in\M$ with $\M$-index $-\alpha$ ($\alpha>0$).
Let $X_{1:n}\leq\cdots\leq X_{n:n}$ be its order statistics.
If \eqref{kn} is satisfied, then we have
\begin{equation}\label{eq:20150328010}
\frac{1}{\hat{\alpha}}\overset{P}{\to}\frac{1}{\alpha}\quad\textrm{as \quad $n\to\infty$.}
\end{equation}
\end{teo}

\begin{proof}
Since
$
\displaystyle \limx\frac{\log\,\barF(x)}{\log(x)}=-{\alpha}\textrm{,}
$
by applying Theorem \ref{teoCK},
%
combining it with Lemma \ref{lem:main:lem3.2.1} gives
$$
\limn\frac{\log(X_{n-k:n})}{\log\left(\frac{n}{k}\right)}
\ =\ \limn\frac{\log(X_{n-k:n})}{\log\,\barF(X_{n-k:n})}\times\frac{\log\,\barF(X_{n-k:n})}{\log\left(\frac{n}{k}\right)}\ \overset{a.s.}{=}\ 
-\frac{1}{\alpha}\times\limn\frac{\log\,\barF(X_{n-k:n})}{\log\left(\frac{n}{k}\right)}\textrm{.}
$$
%
Noticing that $-\log\,\barF(X)$ is a r.v. following a standard exponential distribution $G$ (see e.g. \cite{GalambosKotz}, page 18), 
$-\log\,\barF(X_1)$, \ldots, $-\log\,\barF(X_n)$ is a sample of i.i.d. r.v.s following $G$, 
and $-\log\,\barF(X_{1:n})$, \ldots, $-\log\,\barF(X_{n:n})$ is its order statistics (see e.g. \cite{GalambosKotz}, page 20),
then applying Lemma \ref{lem:20150324:001} 
provides that
$$
\limn\frac{1}{\hat{\alpha}}
\ =\ \limn\frac{\log(X_{n-k:n})}{\log\left(\frac{n}{k}\right)}
\ \overset{P}{=}\ \frac{1}{\alpha}\textrm{.}
$$
This concludes the proof.
\end{proof}

\begin{rmq}\label{rmk01}
Theorem \ref{prop:20150324:001} holds for some variants of $\hat{\alpha}$, for instance
\begin{eqnarray*}
\hat{\alpha}_1 & = & \frac{C_1+\log\left(n\big/k\right)}{C_2+\log(X_{n-k:n})}\textrm{,\quad for given constants $C_1,C_2\in\Rset$, for enough large $n$.} \\
\hat{\alpha}_2 & = & \frac{1}{k_2-k_1+1}\sum_{k=k_1}^{k_2}\frac{\log\left(n\big/k\right)}{\log(X_{n-k:n})}\textrm{,\quad for given $1\leq k_1<k_2$ being $k_2-k_1$ a constant.}
\end{eqnarray*}
\end{rmq}

\subsubsection{Strong consistency}
\label{strongconsistency}

Lemma \ref{lem:20150324:001} may be strengthened under an additional hypothesis using a result by 
Davis and Resnick \cite{DavisResnick1984}.

\begin{lem}[Lemma 5.1 given by Davis and Resnick (1984) \cite{DavisResnick1984}]\label{lem:20150324:001dr}
Let $E_1$, \ldots, $E_n$ be i.i.d. r.v.s with standard exponential distribution $F$.
Let $E_{1:n}\leq\cdots\leq E_{n:n}$ be its order statistics.
Then, if $k=k(n)=\lfloor n^\delta\rfloor$, with $0<\delta<1$, then we have 
$$
E_{n-k:n}-\log\left(\frac{n}{k}\right)\overset{a.s.}{\to}0\quad\textrm{as $n\to\infty$.}
$$
\end{lem}

Note that $k$ defined in Lemma \ref{lem:20150324:001dr} satisfies \eqref{kn}. 
This lemma is a refinement of Lemma \ref{lem:20150324:001}.
Using similar arguments as in the proof of Theorem \ref{prop:20150324:001} and applying Lemma \ref{lem:20150324:001dr} instead of Lemma \ref{lem:20150324:001}, we obtain: 

\begin{teo}[{Strong consistency of $\hat{\alpha}$}]\label{prop:20150324:001bis}
Let $X$, $X_1$, \ldots, $X_n$ be i.i.d. r.v.s with continuous distribution $F$ such that $\barF\in\M$ with $\M$-index $-\alpha$ ($\alpha>0$).
Let $X_{1:n}\leq\cdots\leq X_{n:n}$ be its order statistics.
If $k=k(n)=\lfloor n^\delta\rfloor$, with $0<\delta<1$, then we have
\begin{equation}\label{eq:20150328012}
\frac{1}{\hat{\alpha}}\overset{a.s.}{\to}\frac{1}{\alpha}\quad\textrm{as\quad $n\to\infty$.}
\end{equation}
\end{teo}

\begin{rmq}\label{rmk01bis}
Theorem \ref{prop:20150324:001bis} holds for the variants of $\hat{\alpha}$, $\hat{\alpha}_1$ and $\hat{\alpha}_2$, presented in Remark \ref{rmk01}.
\end{rmq}

\subsection{Asymptotic normality}
\label{normalasymptotic}

In order to analyze the asymptotic normality of $\hat{\alpha}$ we will assume, as in Hall \cite{Hall1982},
a continuous distribution $F$ with $\barF$ satisfying, for given $\alpha,C>0$, as $x\to\infty$,
\begin{equation}\label{eq:20150325:001}
\barF(x) = Cx^{-\alpha}\big(1+o(x^{-\beta})\big)\textrm{.}
\end{equation}
%
Note that $\barF$ is RV with tail index $\alpha$, so $\barF\in\M$ with $\M$-index $-\alpha$.
%

In order to take into account $C$ involved in \eqref{eq:20150325:001}, we modify a bit $\hat{\alpha}$ by defining it as
\begin{equation}\label{eq:20150316:010bis}
\hat{\alpha}=\frac{\log\left(n\big/k\right)+\log(C)}{\log(X_{n-k:n})}\textrm{.}
\end{equation}
Note that this estimator corresponds to the variant $\hat{\alpha}_1$ when taking $C_1=\log(C)$ and $C_2=0$.


Our main result concerning the asymptotic normality of $\hat{\alpha}$ follows.
\begin{teo}[{Asymptotic normality of $\hat{\alpha}$}]\label{teo:20150329:001}
Assume that $F$ is a continuous distribution on $\Rset^+$ 
satisfying $\barF(x) = Cx^{-\alpha}\big(1+o(x^{-\beta})\big)$ as $x\to\infty$, for positive constants $C$, $\alpha$ and $\beta$.
Let $X_1$, \ldots, $X_n$ be i.i.d. r.v.s with distribution $F$, and let $X_{1:n}\leq\cdots\leq X_{n:n}$ be its order statistics.
If \eqref{kn} is satisfied, then
\begin{equation}\label{eq:20150329:002}
\sqrt{k}\left(\log\left(\frac{n}{k}\right)+\log(C)\right)\left(\frac{1}{\hat{\alpha}}-\frac{1}{\alpha}\right)\overset{d}{\to}\mathcal{N}\big(0,\alpha^{2}\big)\textrm{.}
\end{equation}
\end{teo}

I am indebted to John Einmahl for having pointed out an error in an early formulation of (\ref{eq:20150329:002}).

\begin{proof}
In order to deduce the inverse function of $\barF(x)$, we rewrite it as, as $x\to\infty$,
$$
\barF(x) = \left(C^{1/\alpha} x^{-1}\exp\left(\alpha^{-1} o(x^{-\beta})\right)\right)^{\alpha}.
$$
Then, the inverse function of $\barF(x)$ looks like, as $t\to0^+$,
$$
{\barF}^{\,-1}(t) = \big(Ct^{-1}\exp\big(\alpha\,g(t)\big)\big)^{1/\alpha},
$$
where $g$ is a function to be identified.
To this aim, introducing $x={\barF}^{\,-1}(t)$ in the last expression of $\barF(x)$, we have
$$ 
t = \Big[C^{1/\alpha} \Big\{\big(Ct^{-1}\exp\big(\alpha\,g(t)\big)\big)^{1/\alpha}\Big\}^{-1}\exp\Big\{\alpha^{-1} o\Big(\big(\big(Ct^{-1}\exp\big(\alpha\,g(t)\big)\big)^{1/\alpha}\big)^{-\beta}\Big)\Big\}\Big]^{\alpha}.
$$
Then, applying the logarithm function gives, as $t\to0^+$,
$$
0=-g(t)+o\Big(t^{\beta/\alpha}\exp\big(\beta g(t)\big)\Big).
$$
This necessarily implies that $g(t)\to0$ as $t\to0^+$, and we then deduce that $g(t)=o(t^{\beta/\alpha})$.
Hence, 
as $t\to0^+$,
%
\begin{equation}\label{eq:20150325:003}
\barF^{\,-1}(t)=C^{1/\alpha}t^{-1/\alpha}e^{o(t^{\beta/\alpha})}\textrm{.}
\end{equation}
Let $E_1$, \ldots, $E_n$ be independent standard exponential r.v.s.
Define
$$
S_{i}=\sum_{j=1}^{i}\frac{E_{n-j+1}}{n-j+1}
=\sum_{j=1}^{i}\frac{E_{n-j+1}-1}{n-j+1}+\log\left(\frac{n}{i}\right)+O(i^{-1})\textrm{.}
$$
Then, by R\'{e}nyi's representation of order statistics (see \cite{Renyi1953} and e.g. \cite{David}, page 21), we may write 
$X_{n-k:n}=\barF^{\,-1}\big(e^{-S_{n-k}}\big)$ for $k$ = 1, \ldots, $n-1$.
It now follows from \eqref{eq:20150325:003} that
%
\begin{eqnarray}
\lefteqn{\log(X_{n-k:n})\quad=\quad\frac{1}{\alpha}\,\log C+\frac{1}{\alpha}\,S_{n-k}+o_p(e^{-\beta/\alpha S_{n-k}})} \nonumber \\
 & = & \frac{1}{\alpha}\,\log C+\frac{1}{\alpha}\,\sum_{j=1}^{n-k}\frac{E_{n-j+1}-1}{n-j+1}+\frac{1}{\alpha}\,\log\left(\frac{n}{k}\right) 
+\frac{1}{\alpha}\,O(k^{-1})
+o_p\left(\left(\frac{k}{n}\right)^{\beta/\alpha}\left(1+O(k^{-1})\right)^{\beta/\alpha}\right) 
\textrm{.}\label{eq:20150329:020}
\end{eqnarray}

The mean of $\sqrt{k}\alpha^{-1}\sum_{j=1}^{n-k}\left(E_{n-j+1}-1\right)\big/(n-j+1)$ is 0 and its variance
$$
k\alpha^{-2}\sum_{j=k+1}^n\frac{1}{j^2}\sim\alpha^{-2}\textrm{\quad as\quad $n\to\infty$,}
$$
since
$$
-\frac{1}{n+1}+\frac{1}{k+1}\ =\ \int_{k+1}^{n+1}\frac{dx}{x^2}\ \leq\ \sum_{j=k+1}^n\frac{1}{j^2}\ \leq\ \frac{1}{(k+1)^2}+\int_{k+1}^{n}\frac{dx}{x^2}\ =\ \frac{1}{k+1}\left(1+\frac{1}{k+1}\right).
$$

Hence, from \eqref{eq:20150329:020} we have that $\sqrt{k}\,\big(\log\big(X_{n-k:n}\big)-\alpha^{-1}\,\log\big(n\big/k\big)-\alpha^{-1}\,\log C\big)$ is asymptotically $\mathcal{N}\big(0,\alpha^{2}\big)$, and \eqref{eq:20150329:002} then follows.
\end{proof}

\begin{rmq}\label{rmqX1}
For a given $C>0$, one can compute the optimal convergence rate of $\hat{\alpha}$ in \eqref{eq:20150329:002}.
By differentiating $\sqrt{k}\,\big(\log\big(n\big/k\big)+\log(C)\big)$ with respect to $k$ and straightforward computations, one obtains that this optimal rate is of type $C^{-1/2}\,n^{-1/2}$.
Surprisingly this result seems to contradict the best attainable rate of convergence for estimates of shape of regular variation given by Hall and Welsh in 1984 (see \cite{HallWelsh1984}), which was established to be higher than $n^{-1/2}$.

This apparent contradiction motivated our interest in the analysis of reasons to get a rate that is not allowed according to \cite{HallWelsh1984}.
That study presented in \cite{NN2015HallWelsh} shows that the best attainable rate of convergence given by Hall and Welsh can be improved, allowing rates lower than $n^{-1/2}$.
So, such contradiction does not exist.

\end{rmq}

\begin{rmq}\label{rmqX2}
Nevertheless the rate of convergence of $\hat{\alpha}$ is lower than those of well-known estimators of $\alpha$, showing thus a better performance than its competitors, its dependence on $C$ could deteriorate its performance, for instance if $C$ takes values lower than 1.
But fortunately eventual deteriorations in the rate would be mitigated by large sample sizes.
However, if $C<1$, a bad performance could be evidenced in practice since large sample sizes are not often available.

\end{rmq}

%
%
%

\section{Simulation study and real applications}
\label{applications}

In this section we provide simulated and real examples to illustrate how the proposed estimator $\hat{\alpha}$ (see \eqref{eq:20150316:010bis}) works.
Also, comparisons of this estimator with some existing ones are shown.

Among other tail index estimators to be applied in this paper are those given by Hill 
and Dekkers, Einmahl, and de Haan.
Nice surveys on these estimators and their properties can be found in de Haan and Ferreira \cite{deHaanFerreira} and Embrechts et al. \cite{embrechts1997}.

For what follows, let $X_1$, \ldots, $X_n$ a sample of i.i.d. r.v.s following a common distribution $F$, and $X_{1:n}\leq\cdots\leq X_{n:n}$ its order statistics.
In order to compare our estimator with some existing ones, we assume that the tail of $\barF$ is RV with  tail index $\alpha$ ($\alpha>0$).

The first well-known estimator to be considered in this section is the estimator of Hill \cite{Hill1975} given in 1975 and defined by
%
$$
\hat{\alpha}_H=\frac{1}{k}\sum_{i=0}^{k-1}\log\left(X_{n-i:n}\right)-\log\left(X_{n-k:n}\right),
$$
%
and the second estimator is the moment estimator given by Dekkers, Einmahl, and de Haan \cite{DekkersEinmahldeHaan1989} in 1989, which is defined by
%
$$
\hat{\alpha}_{M}=M_n^{(1)}+1-\frac{1}{2}\left(1-\frac{\left(M_n^{(1)}\right)^2}{M_n^{(2)}}\right)^{-1}\textrm{,}
$$
%
where $\displaystyle M_n^{(i)}=k^{-1}\sum_{i=0}^{k-1}\big(\log X_{n-i:n}-\log X_{n-k:n}\big)^j$.


A typical way for comparing $\hat{\alpha}$, $\hat{\alpha}_H$ and 
$\hat{\alpha}_M$ 
is plotting their estimates against $k$, keeping $n$ fixed.
For the Hill estimator this is usually called the Hill plot.
We use this approach in the following numerical illustrations to analyze the behaviors of this set of estimators of $\alpha$.

\subsection{Simulation study}

In this subsection we analyze the behavior of estimators when some parameters vary.
To this aim, several samples are built using simulations, by varying $\alpha$ and $C$ and the sample size, $n$.
We consider $\alpha$ = 0.1, 1 and 1.5, $C$ = 0.1, 1 and 10, and $n$ = 1 000, 10 000 and 100 000.
Then we simulate pseudorandom samples of size $n$ of the distribution $F_{\alpha,C}(x)=1-C\,x^{-\alpha}$, $x\geq C^{-\alpha}$.

Plots of estimates of $\alpha$ when varying $\alpha$, $C$, and $n$ are shown.
In each plot $C$ and $n$ are fixed and we allow $\alpha$ to vary, denoting $F$ by $F_1$ when $\alpha=0.1$, by $F_2$ when $\alpha=1$ and by $F_3$ when $\alpha=1.5$.
All these plots include reference lines of the values of $\alpha$ to be estimated.

Plots are organized by estimator of $\alpha$, varying $C$ on row and $n$ on column.
Figure \ref{fig01} corresponds to $\hat{\alpha}_H$, Figure \ref{fig01b} to $\hat{\alpha}_M$ 
and Figure \ref{fig01d} to $\hat{\alpha}$.

Then, over all plots is noted that the convergence of estimators is always reached and it is made in general fastly when $\alpha$ = 0.1.
This situation changes when $\alpha$ = 1 or $\alpha$ = 1.5.
In these two cases is observed that often such convergence is not reached, evidencing a bias.
This is sharper with $\hat{\alpha}_M$.

The Hill estimator, $\hat{\alpha}_H$, shows a better performance than the moment estimator.
This performance improves when $\alpha$ decreases or when $C$ increases.
In fact, considering $C$, the better situation is when $C$ = 10.
Observing sample size, it seems that this variable does not influence the performance of this estimator.

The Dekkers, Einmahl, and de Haan estimator, $\hat{\alpha}_M$, shows biases in its convergence, being these sharp when $\alpha$ = 1 or $\alpha$ = 1.5.
The influences of $C$ and $n$ on this estimator are not clear.


The new estimator, $\hat{\alpha}$, shows the best performance with respect to the other estimators presented in this manuscript.
Moreover, almost always its convergence is reached and is the fastest, without apparently evidencing any bias.
The worst situation is when $C=0.1$ (for $n$ = 1 000 and $n$ = 10 000 and observing $\hat{\alpha}$ for $k<n\times C$) as expected, which was mentioned in Remark \ref{rmqX2}.
But, as was also mentioned there, this issue would be mitigated with large sample sizes.
This is observed when $n$ = 100 000, where a sharp convergence is evidenced.
Also expected, mentioned in Remark \ref{rmqX2}, values of $C$ over 1 would promote faster convergences.
This is observed when $C$ = 10, evidencing in these cases the fastest convergence of the new estimator.
Also, it is appreciated that the increase of $n$ improves the convergence of this estimator.

Note that $\hat{\alpha}(k)=0$ for $k=n\times C$ (supposed integer).
We see that the behavior of this estimator near to $n\times C$ is highly unstable.
However, it is amazing the following fact among the plots of $\hat{\alpha}$ when $C$ = 0.1: for $n\times C<k<n$ with $k$ near to $n$, $\hat{\alpha}$ seems to converge to $\alpha$.
This means that to give an estimate of $\alpha$ using $\hat{\alpha}$, some values of $k$ over $n\times C$ could be used.

From the fact that $\hat{\alpha}(k)=0$ for $k=n\times C$, we deduce that when $C=1$, an unstable behavior of $\hat{\alpha}$ would be expected when $k$ approximates $n$.

\begin{figure}[!ht]
\centering
\subfiguretopcapfalse
\subfigure[$C$ = 0.1, $n$ = 1 000]{
\includegraphics[scale=0.29]{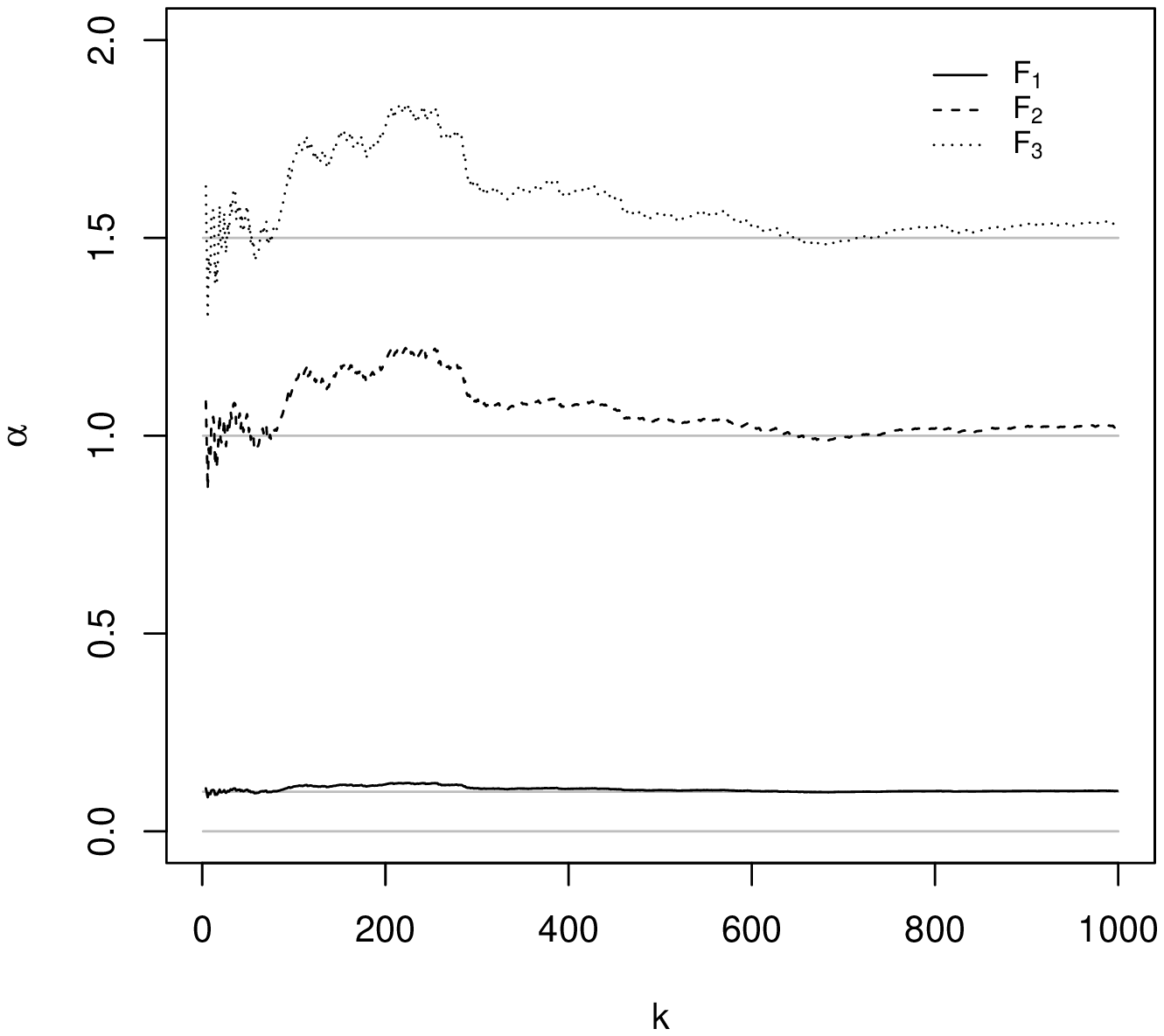}
}
\subfigure[$C$ = 0.1, $n$ = 10 000]{
\includegraphics[scale=0.29]{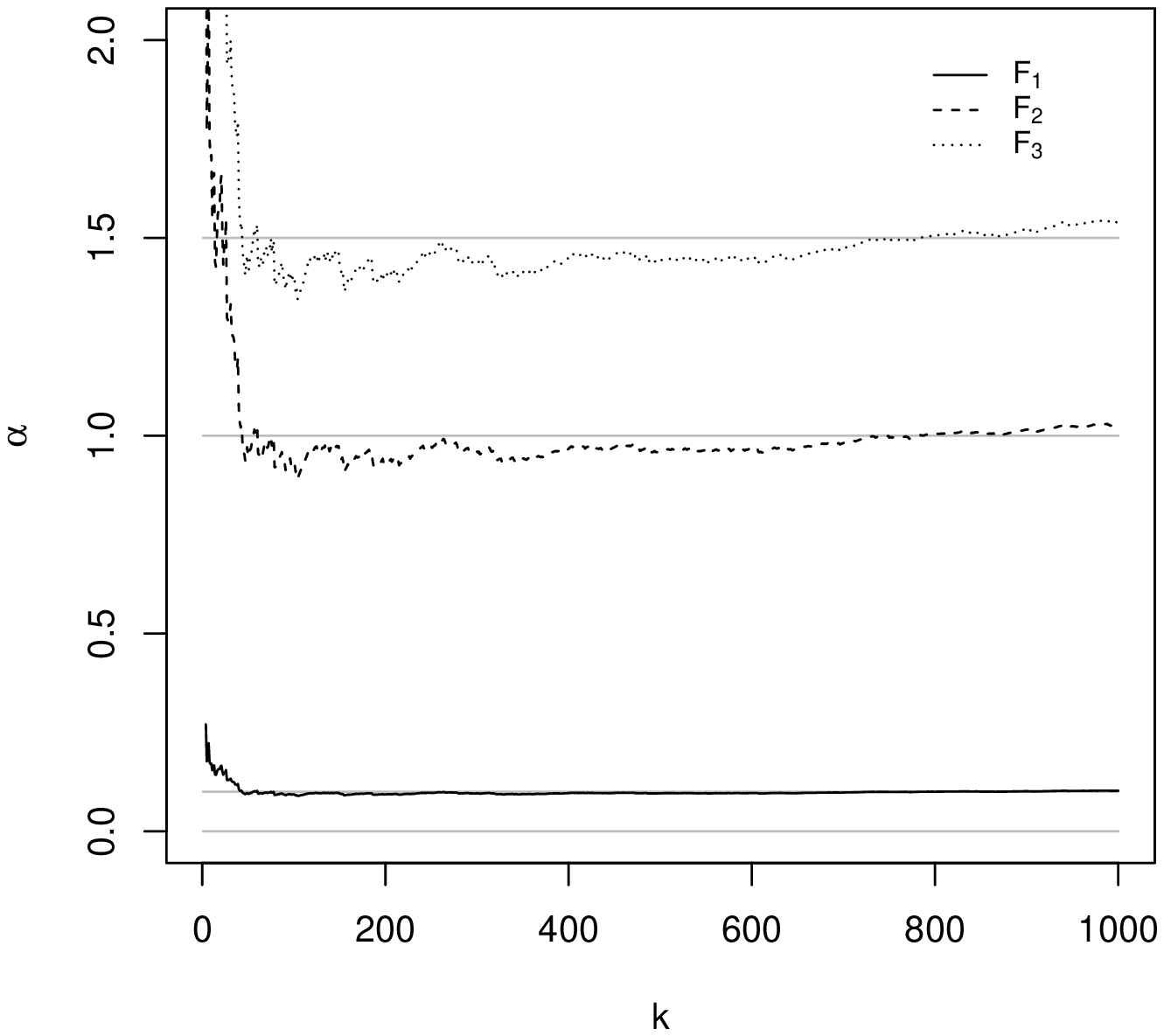}
}
\subfigure[$C$ = 0.1, $n$ = 100 000]{
\includegraphics[scale=0.29]{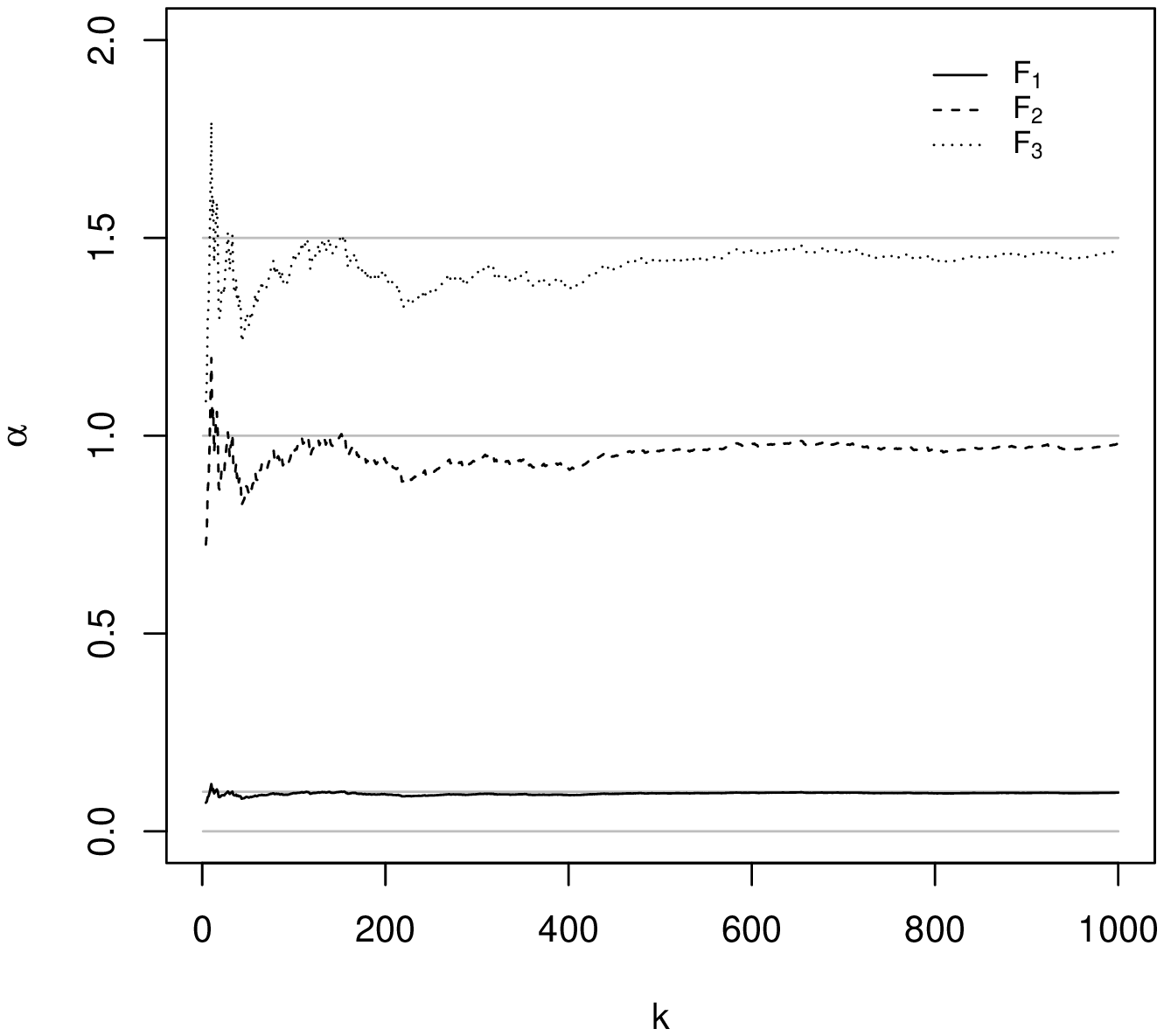}
}
\subfigure[$C$ = 1, $n$ = 1 000]{
\includegraphics[scale=0.29]{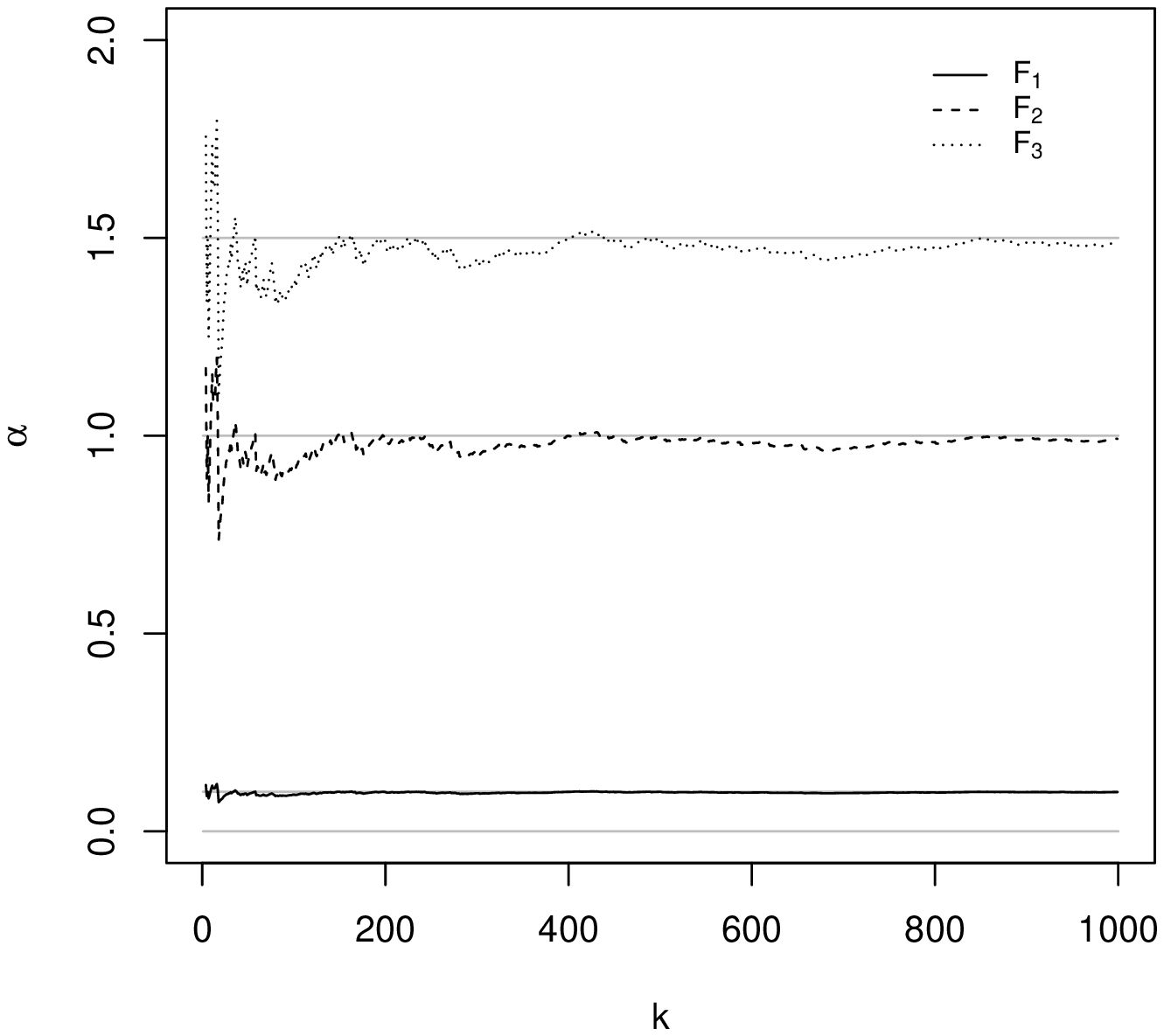}
}
\subfigure[$C$ = 1, $n$ = 10 000]{
\includegraphics[scale=0.29]{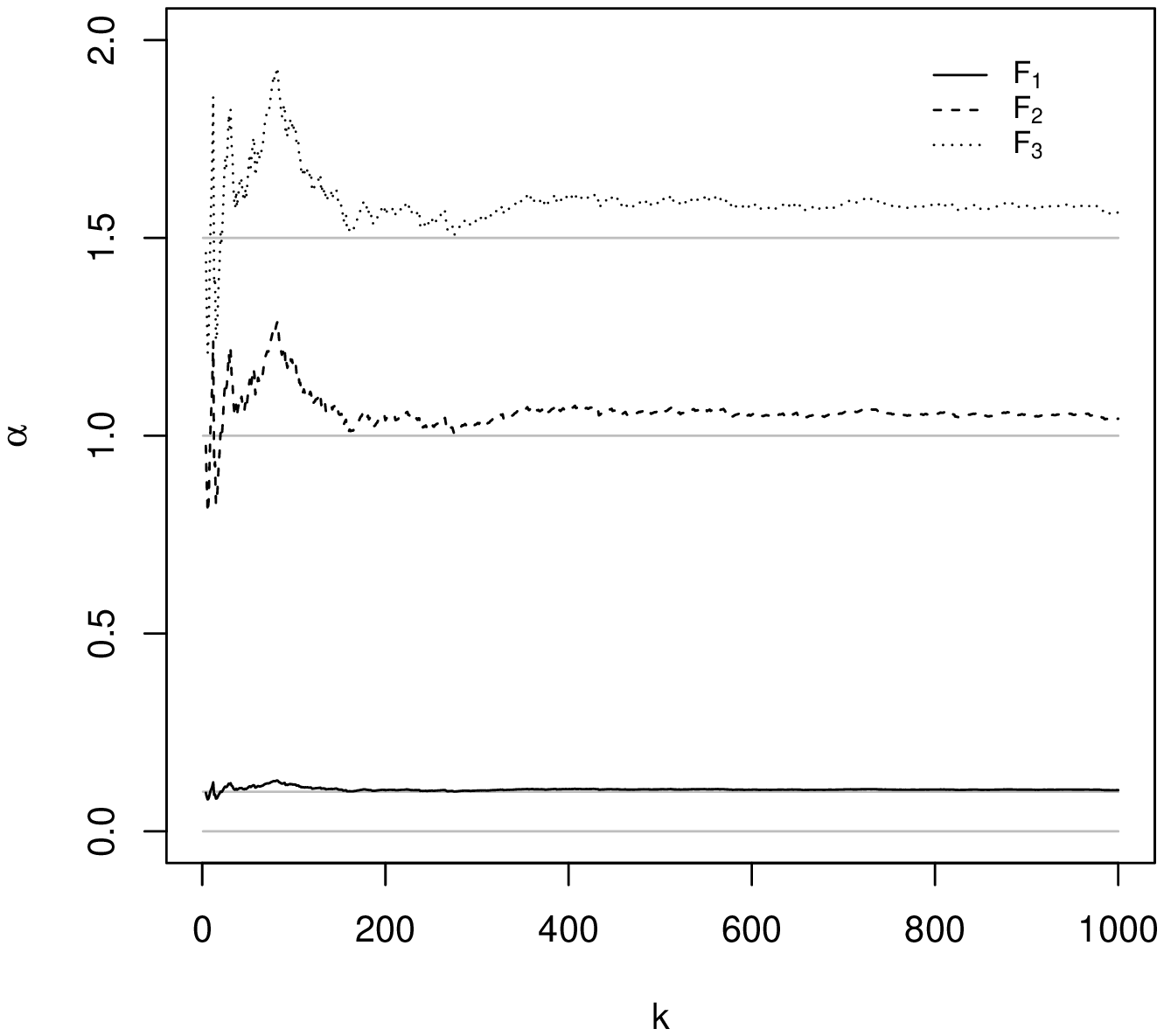}
}
\subfigure[$C$ = 1, $n$ = 100 000]{
\includegraphics[scale=0.29]{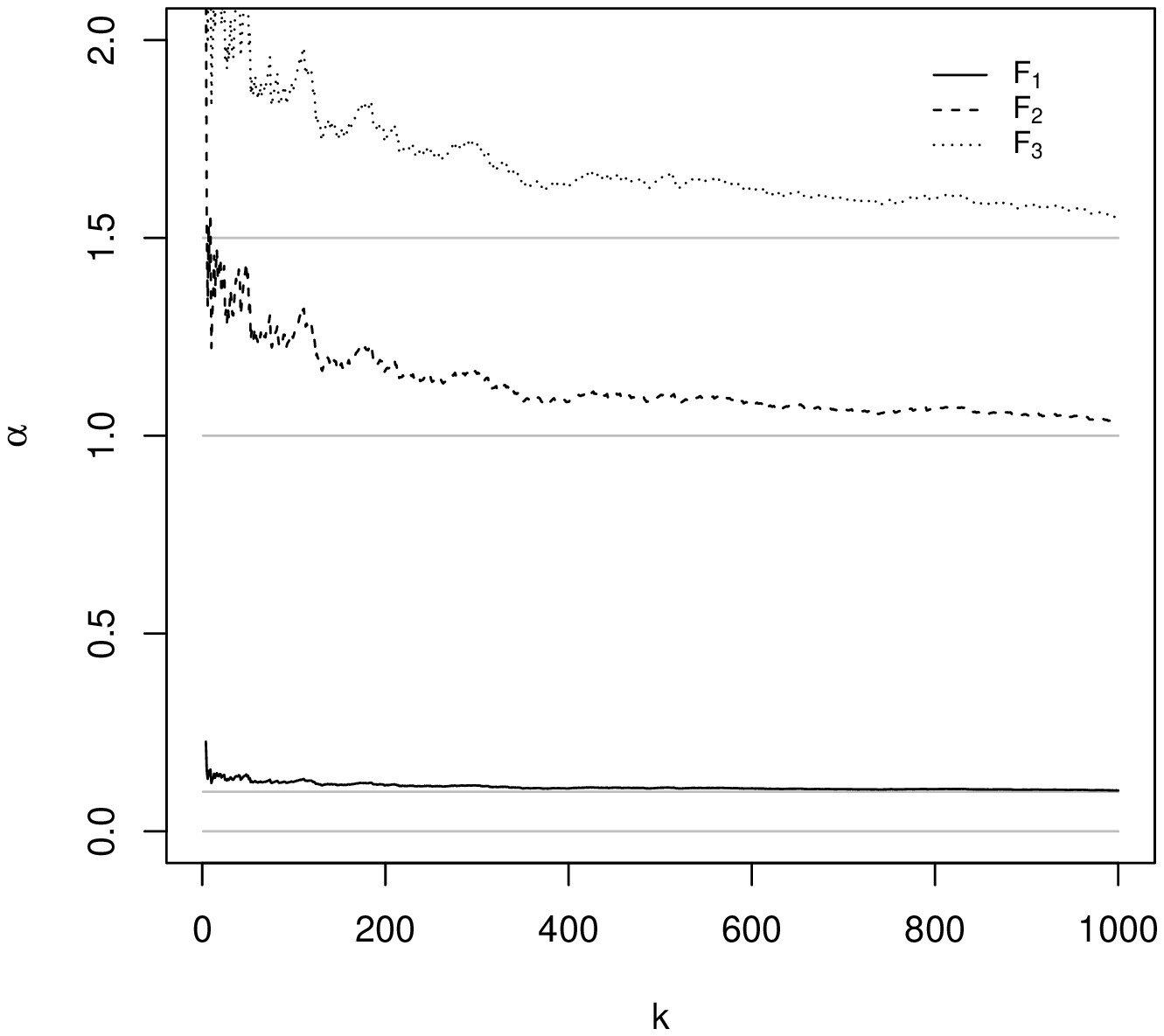}
}
\subfigure[$C$ = 10, $n$ = 1 000]{
\includegraphics[scale=0.29]{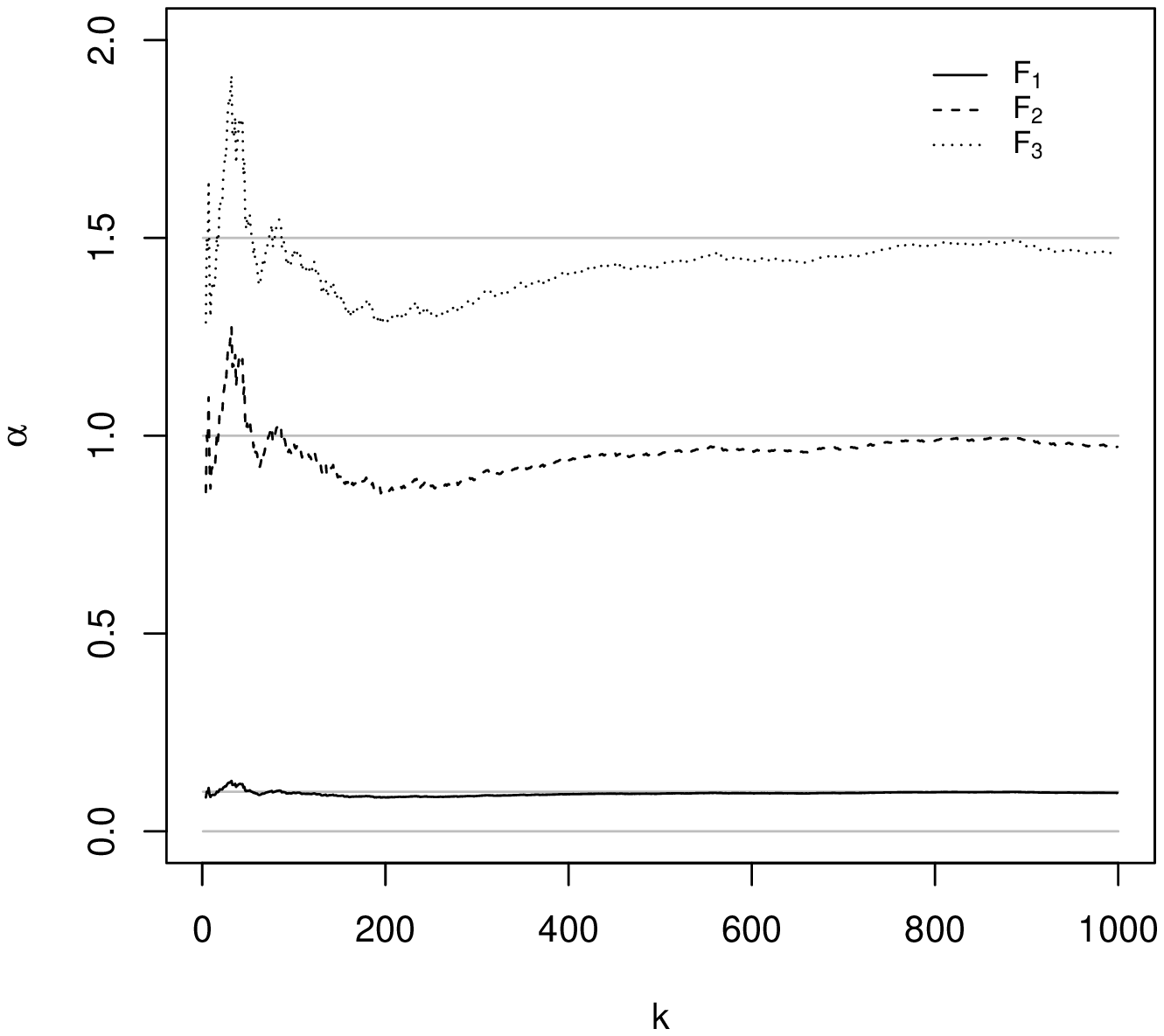}
}
\subfigure[$C$ = 10, $n$ = 10 000]{
\includegraphics[scale=0.29]{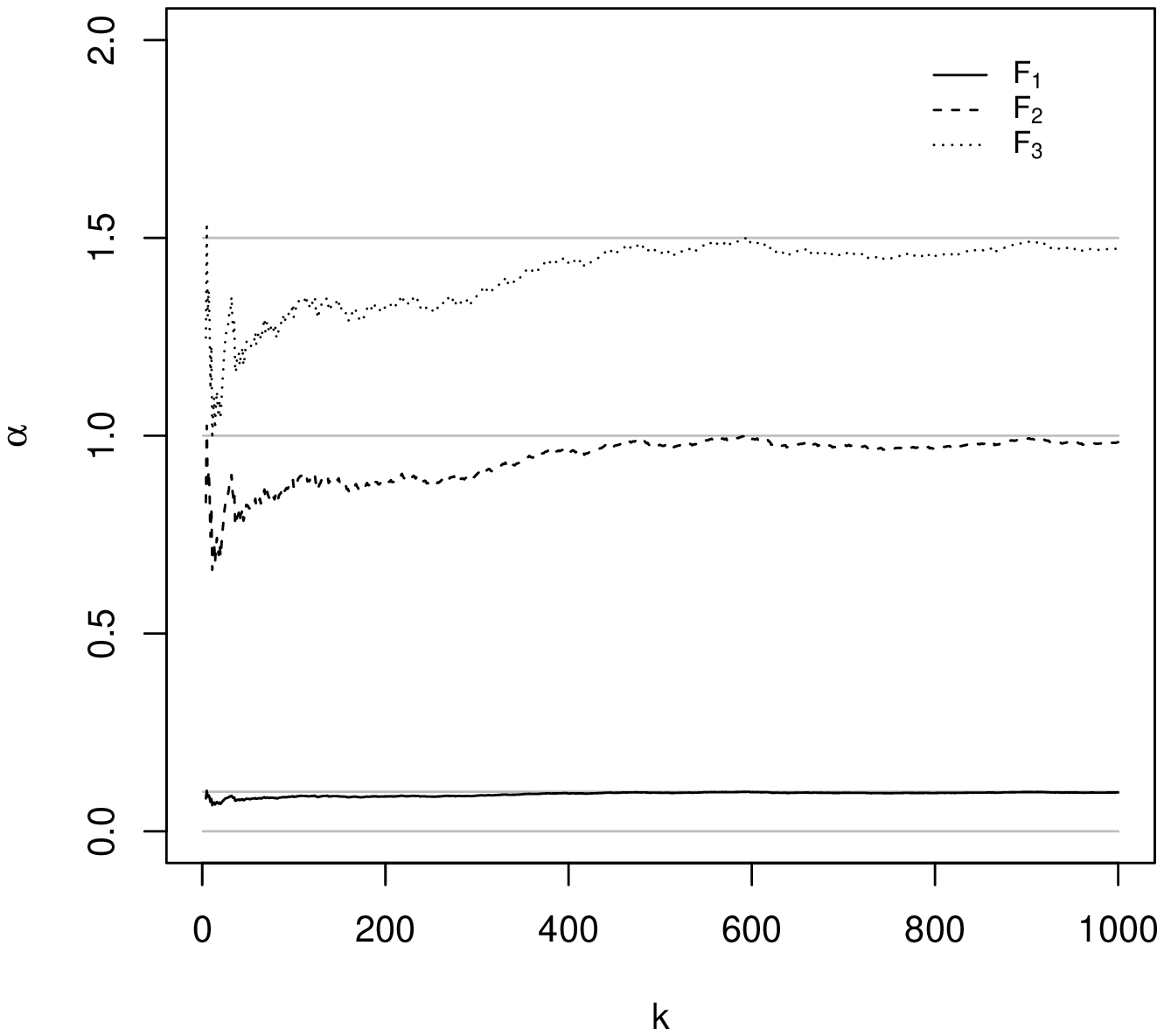}
}
\subfigure[$C$ = 10, $n$ = 100 000]{
\includegraphics[scale=0.29]{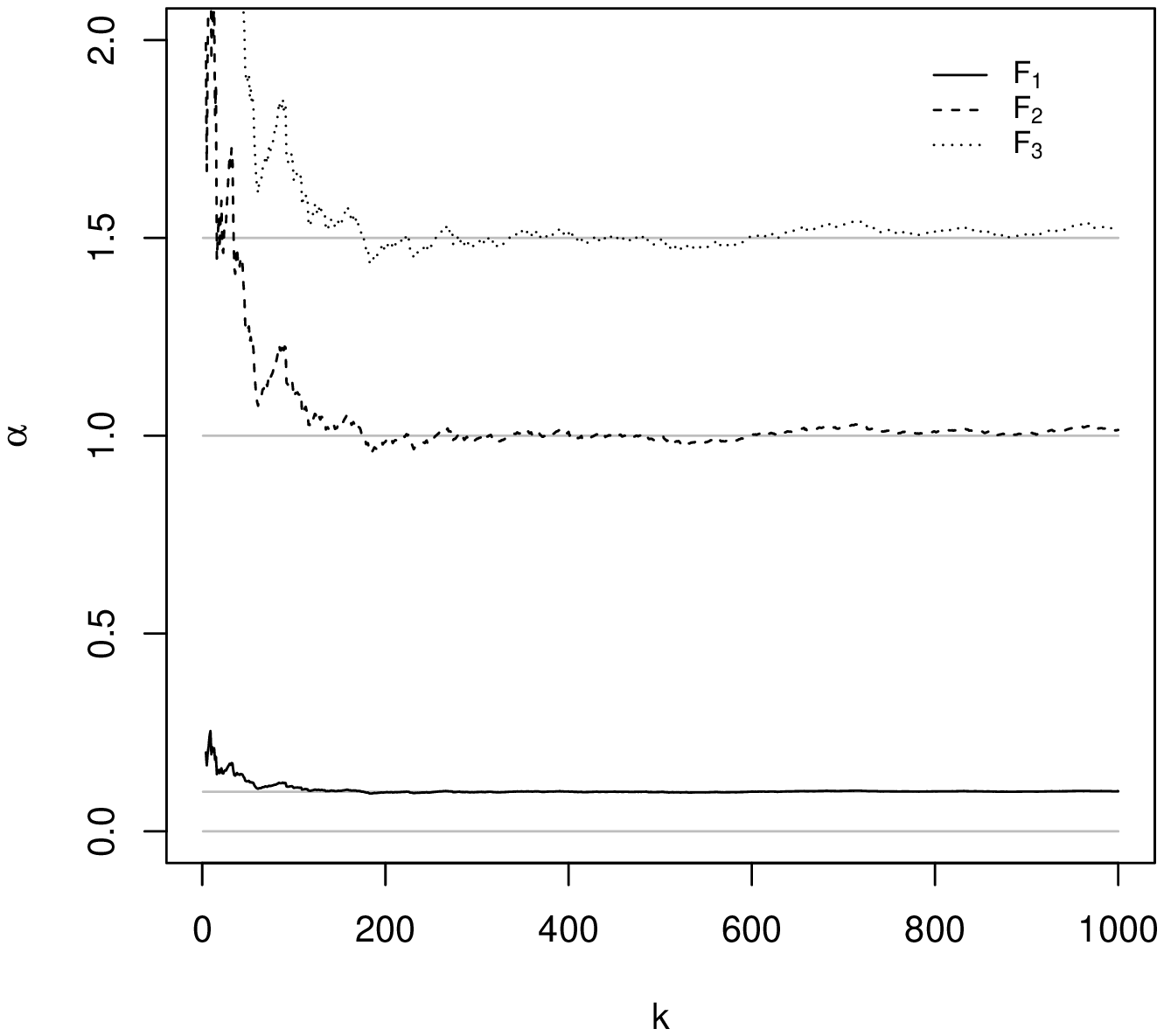}
}
\caption{Plots of $\hat{\alpha}_H$ against $k$, varying $C$ on row and $n$ on column}
\label{fig01}
\end{figure}

\begin{figure}[!ht]
\centering
\subfiguretopcapfalse
\subfigure[$C$ = 0.1, $n$ = 1 000]{
\includegraphics[scale=0.29]{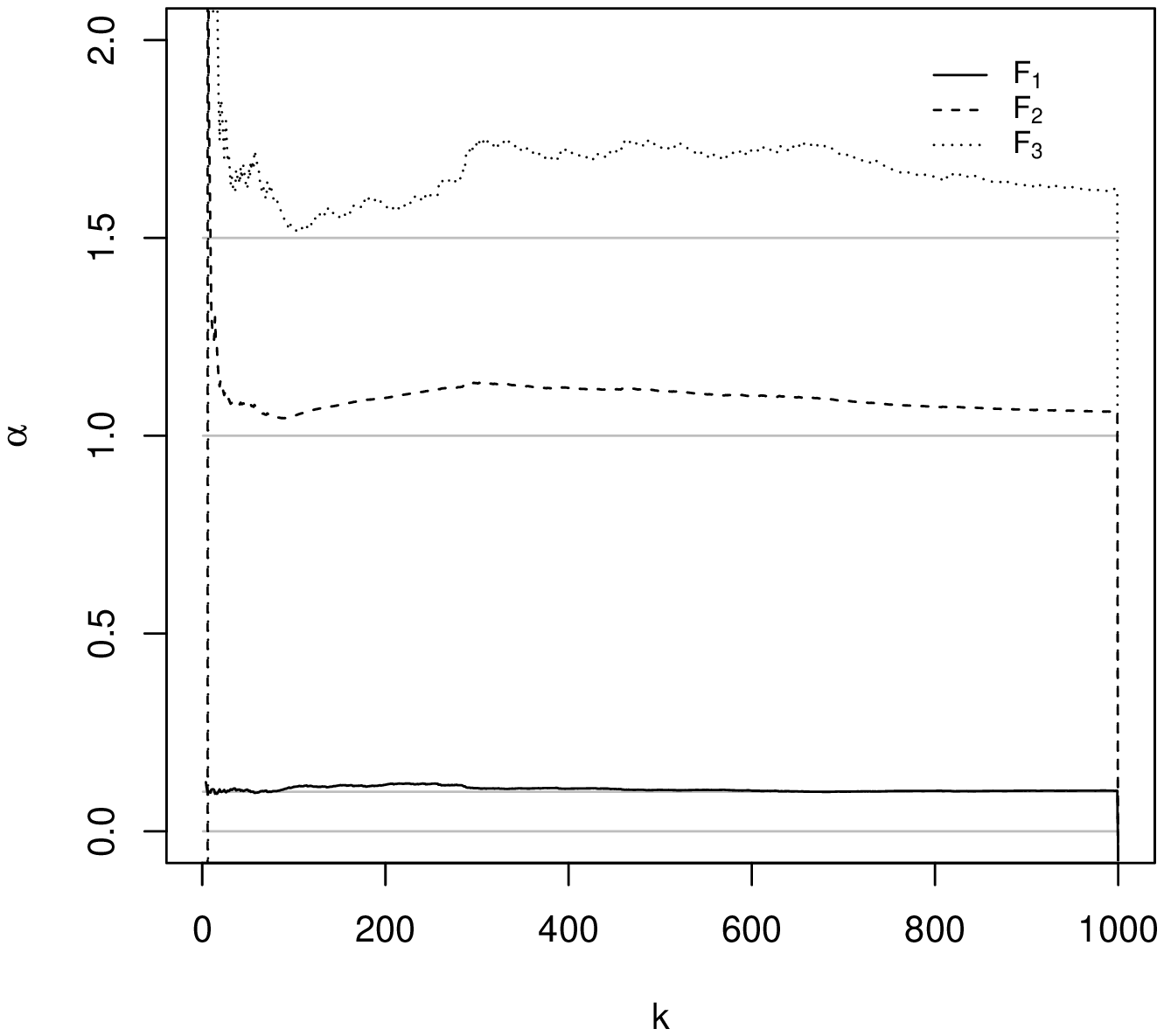}
}
\subfigure[$C$ = 0.1, $n$ = 10 000]{
\includegraphics[scale=0.29]{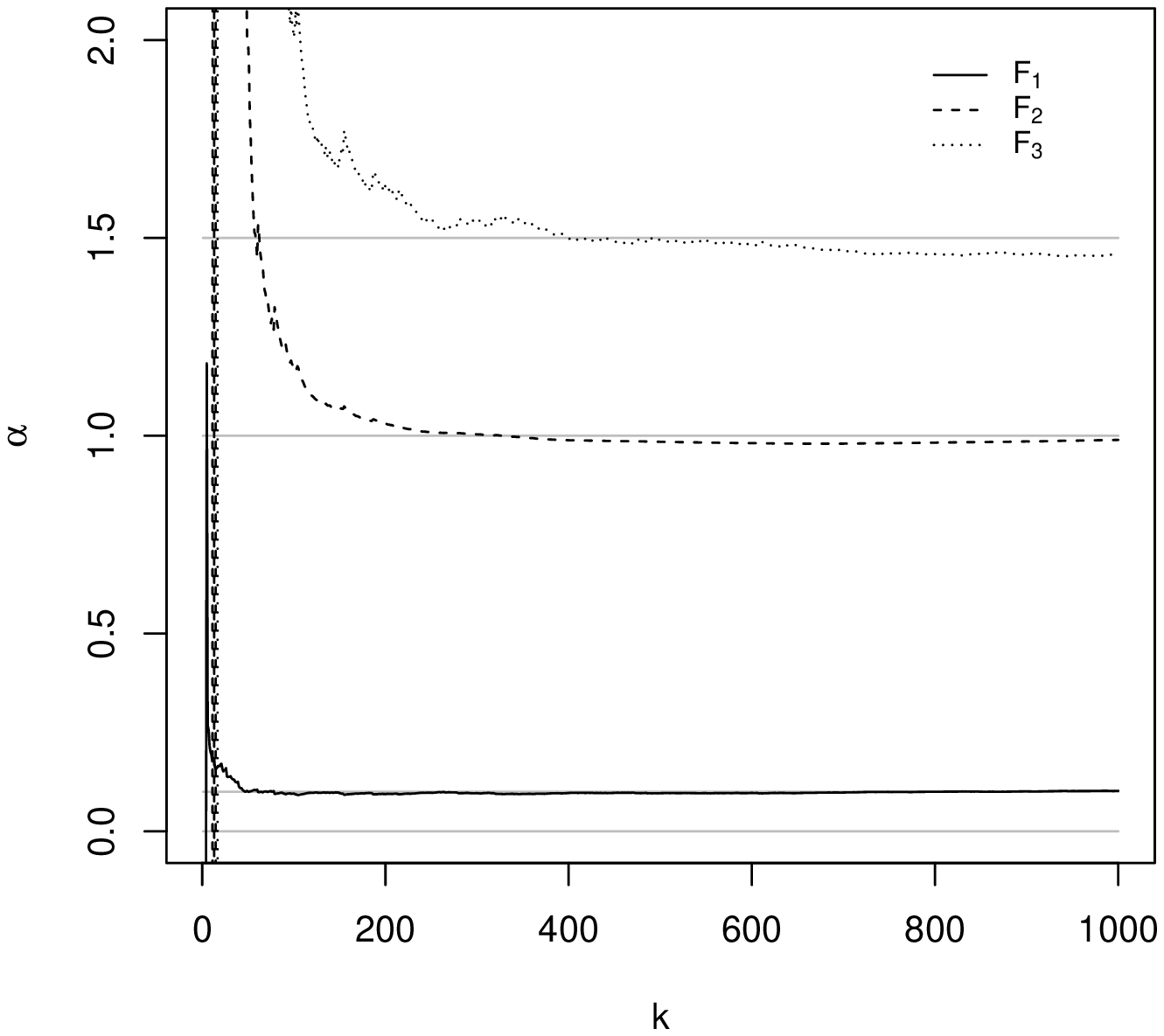}
}
\subfigure[$C$ = 0.1, $n$ = 100 000]{
\includegraphics[scale=0.29]{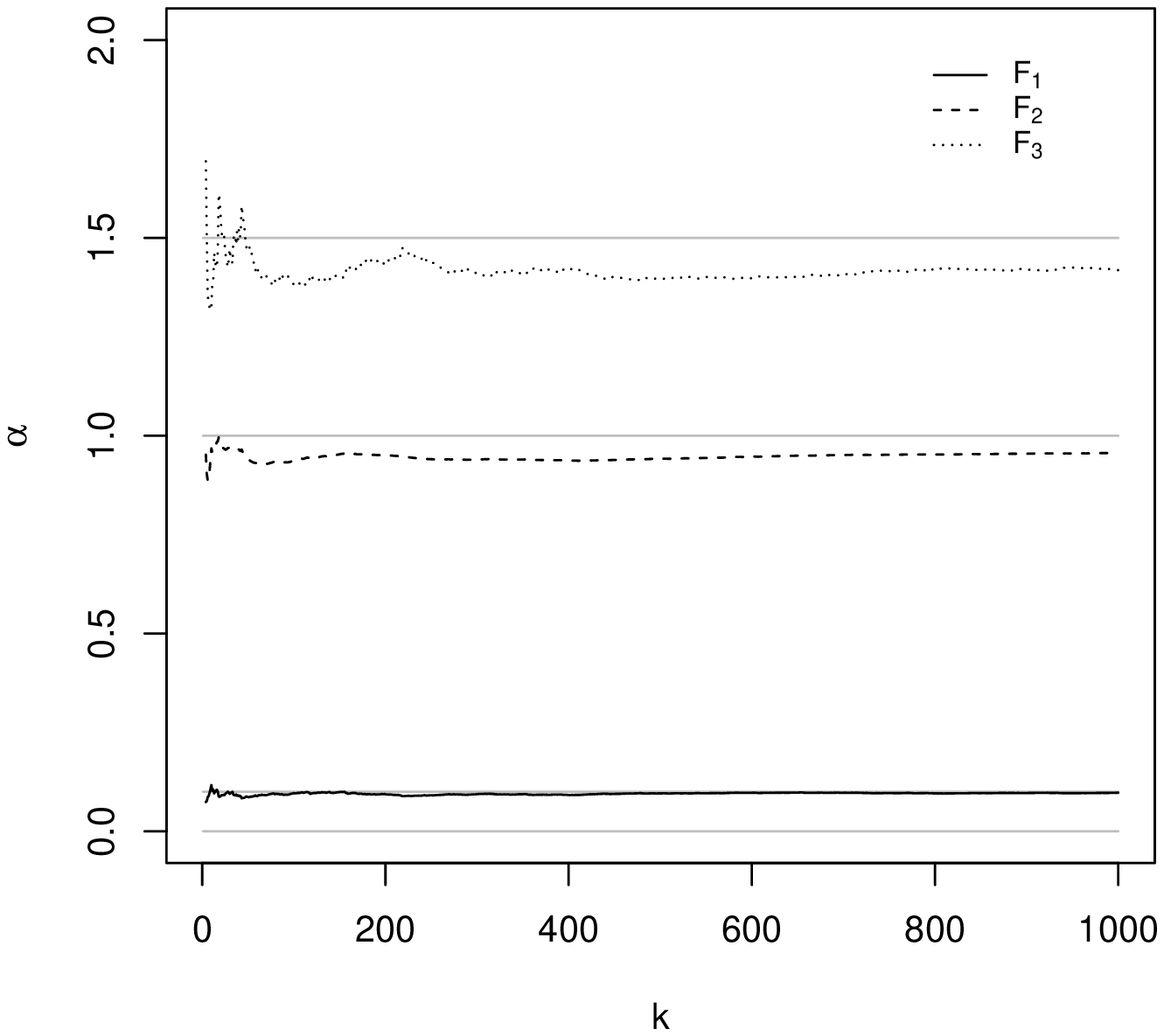}
}
\subfigure[$C$ = 1, $n$ = 1 000]{
\includegraphics[scale=0.29]{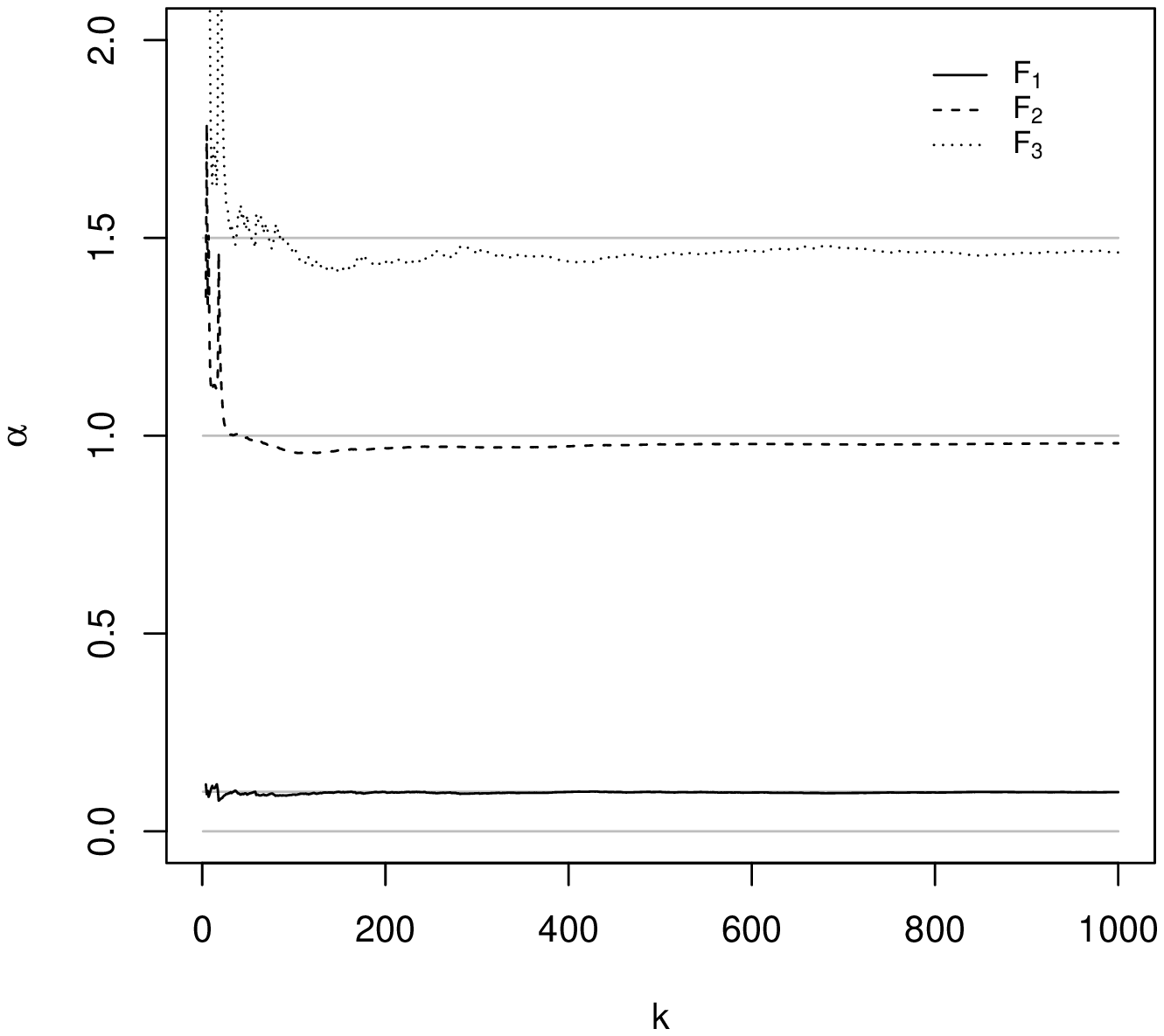}
}
\subfigure[$C$ = 1, $n$ = 10 000]{
\includegraphics[scale=0.29]{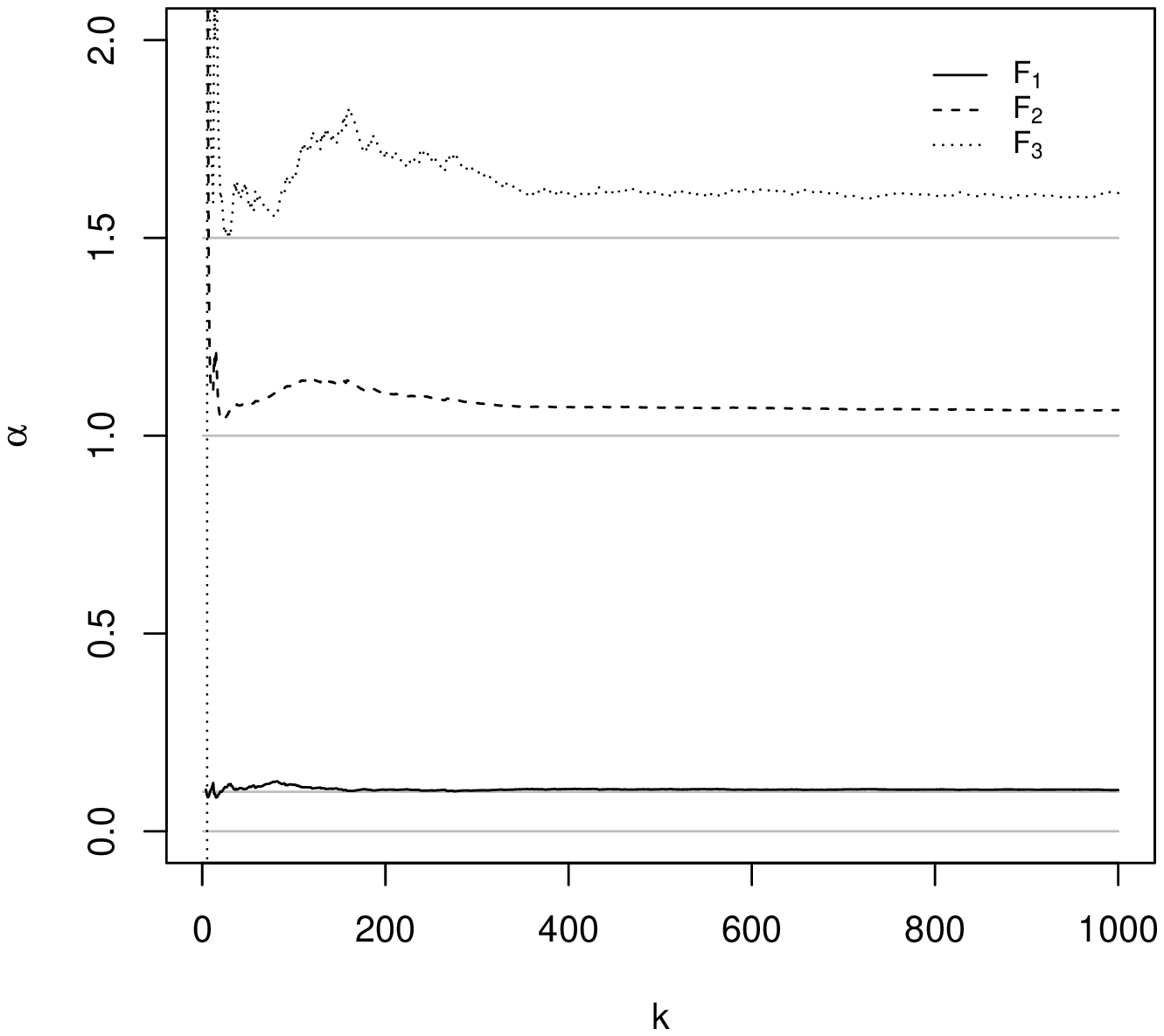}
}
\subfigure[$C$ = 1, $n$ = 100 000]{
\includegraphics[scale=0.29]{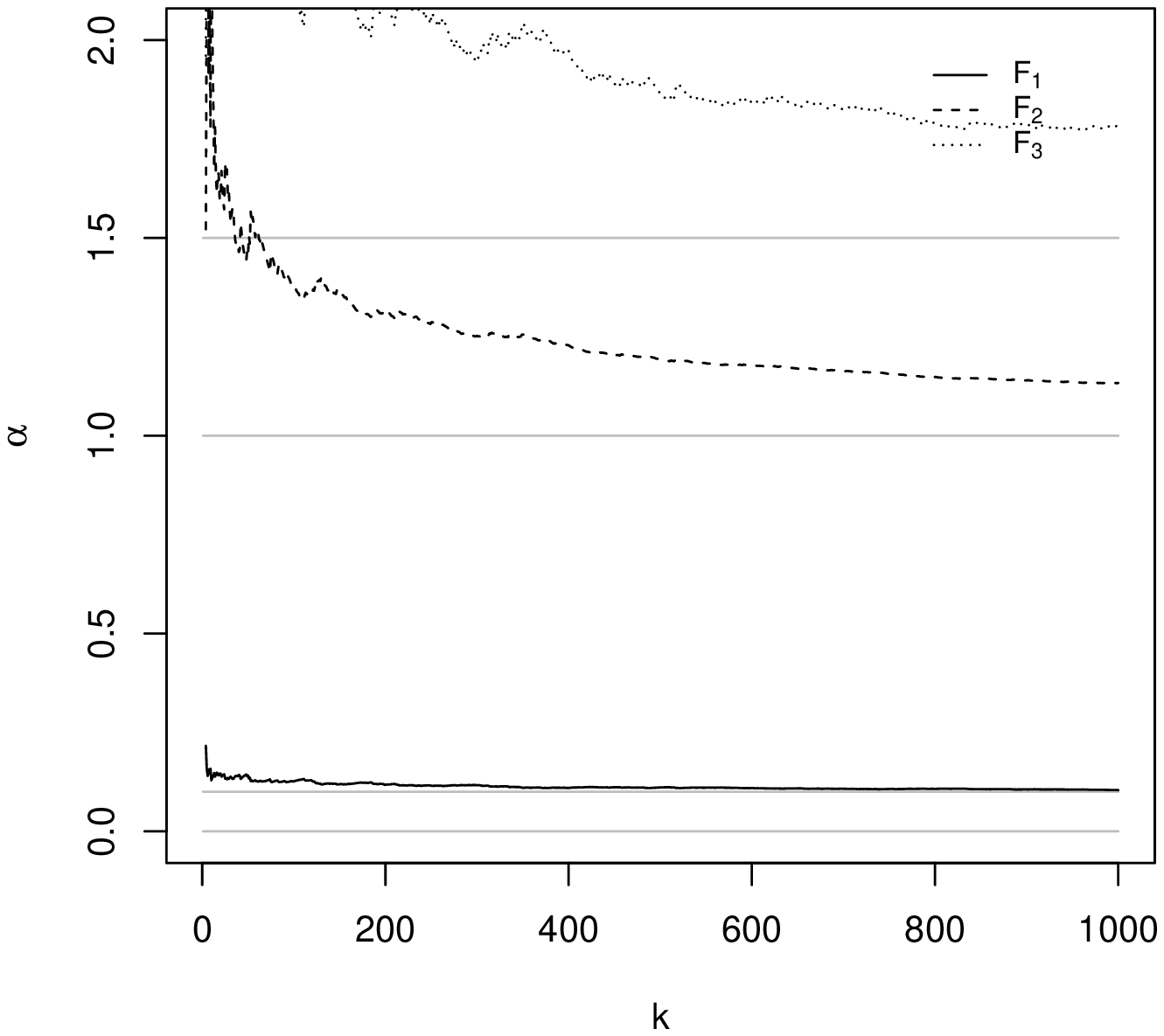}
}
\subfigure[$C$ = 10, $n$ = 1 000]{
\includegraphics[scale=0.29]{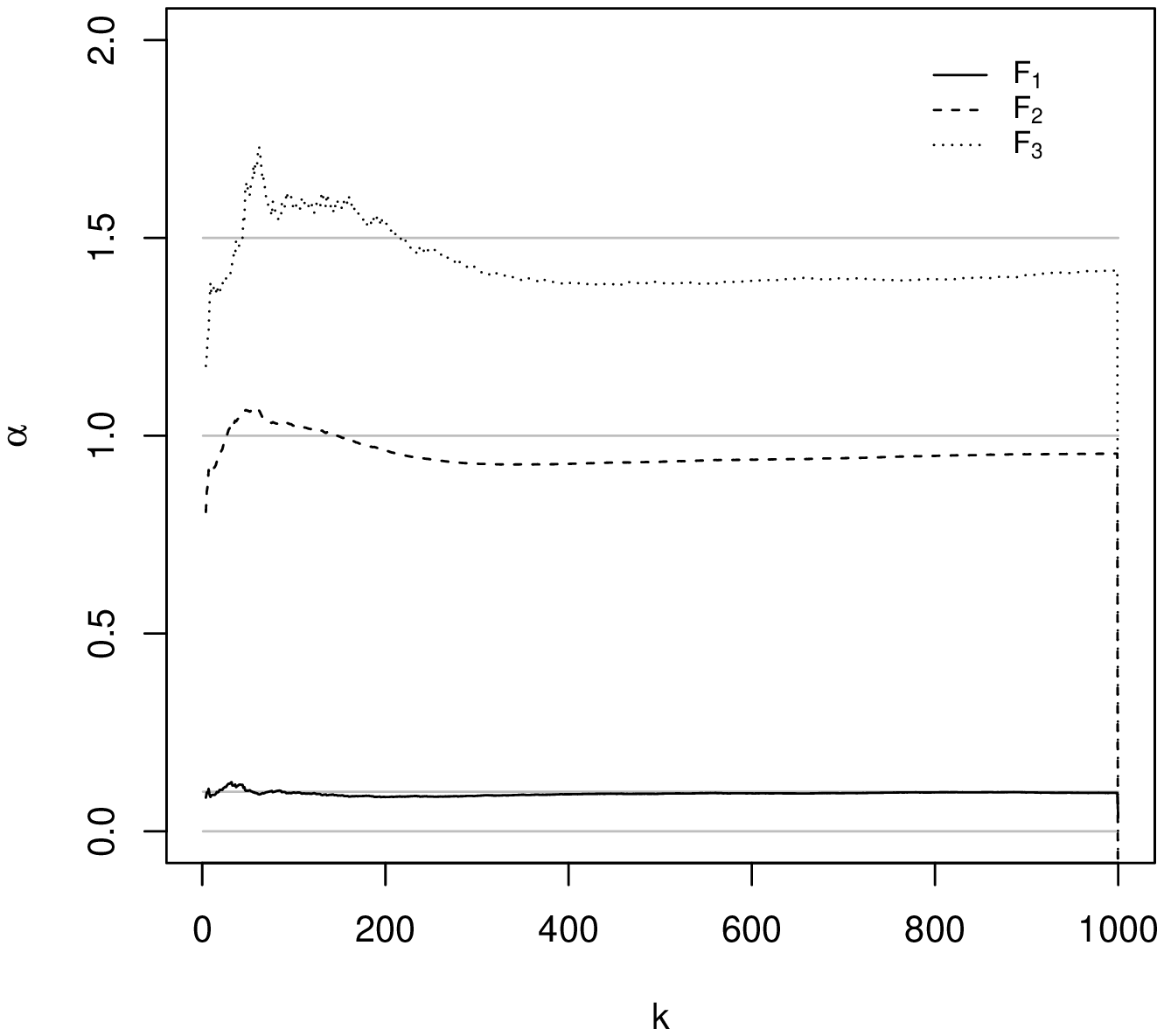}
}
\subfigure[$C$ = 10, $n$ = 10 000]{
\includegraphics[scale=0.29]{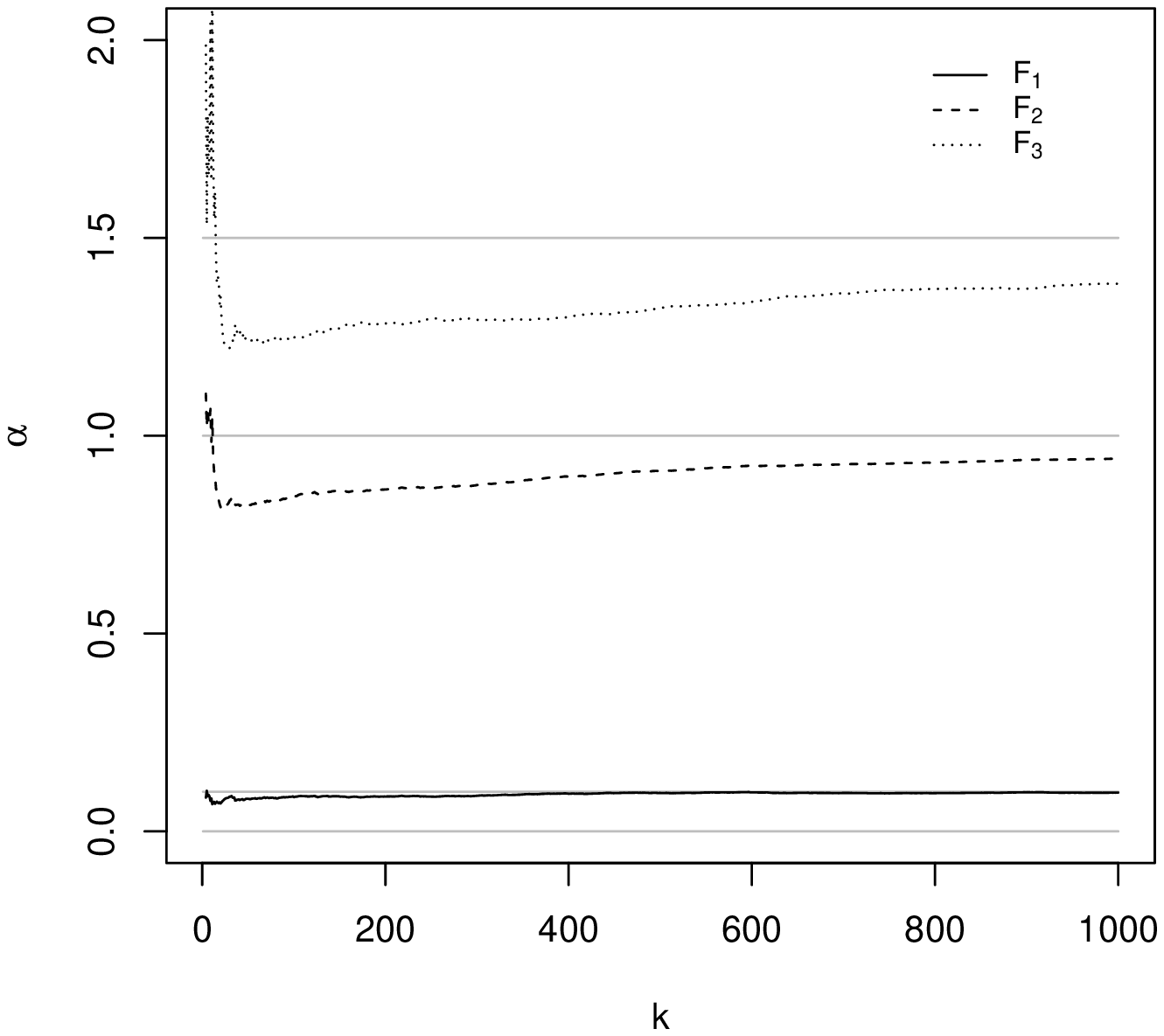}
}
\subfigure[$C$ = 10, $n$ = 100 000]{
\includegraphics[scale=0.29]{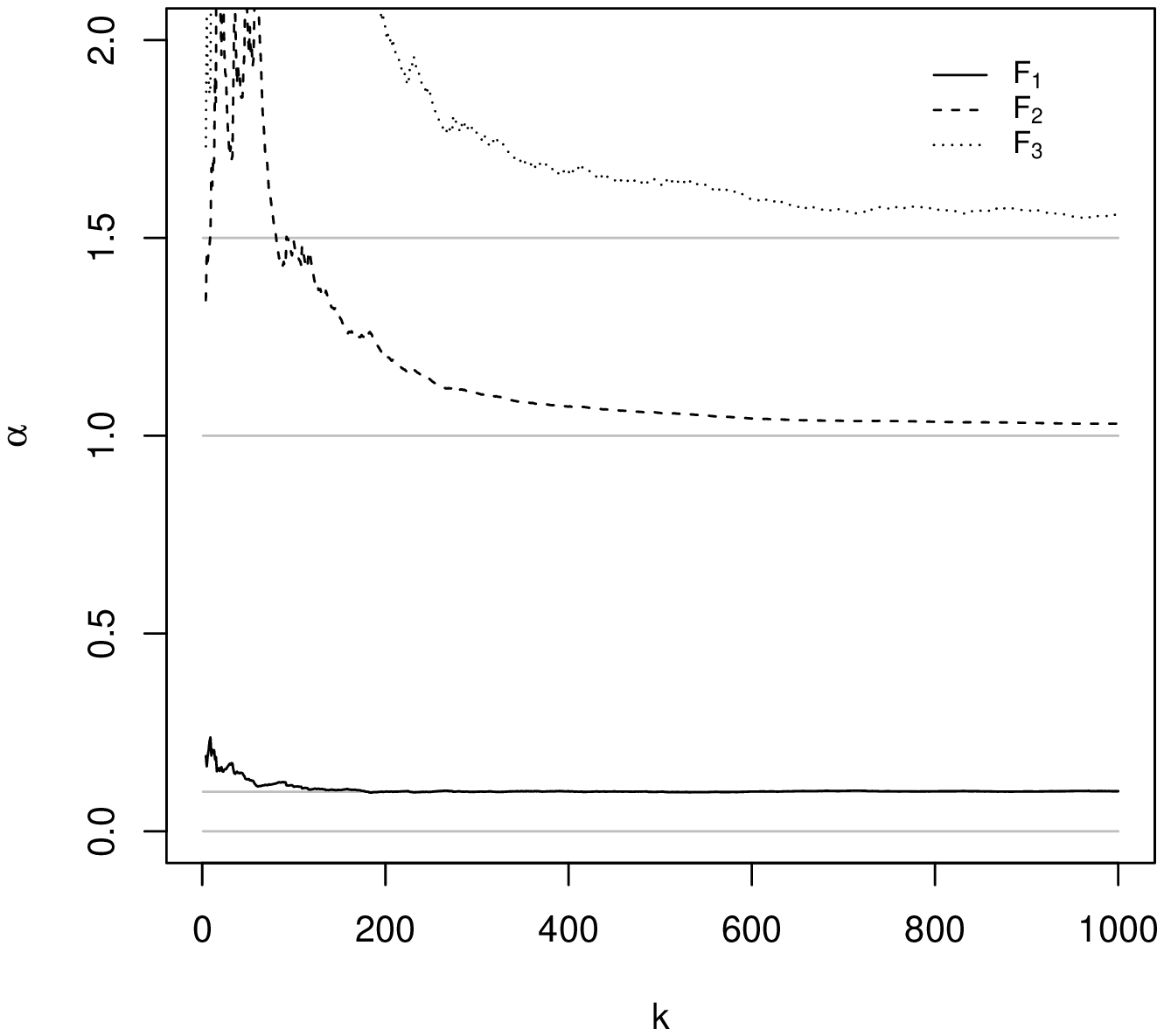}
}
\caption{Plots of $\hat{\alpha}_M$ against $k$, varying $C$ on row and $n$ on column}
\label{fig01b}
\end{figure}

\begin{figure}[!ht]
\centering
\subfiguretopcapfalse
\subfigure[$C$ = 0.1, $n$ = 1 000]{
\includegraphics[scale=0.29]{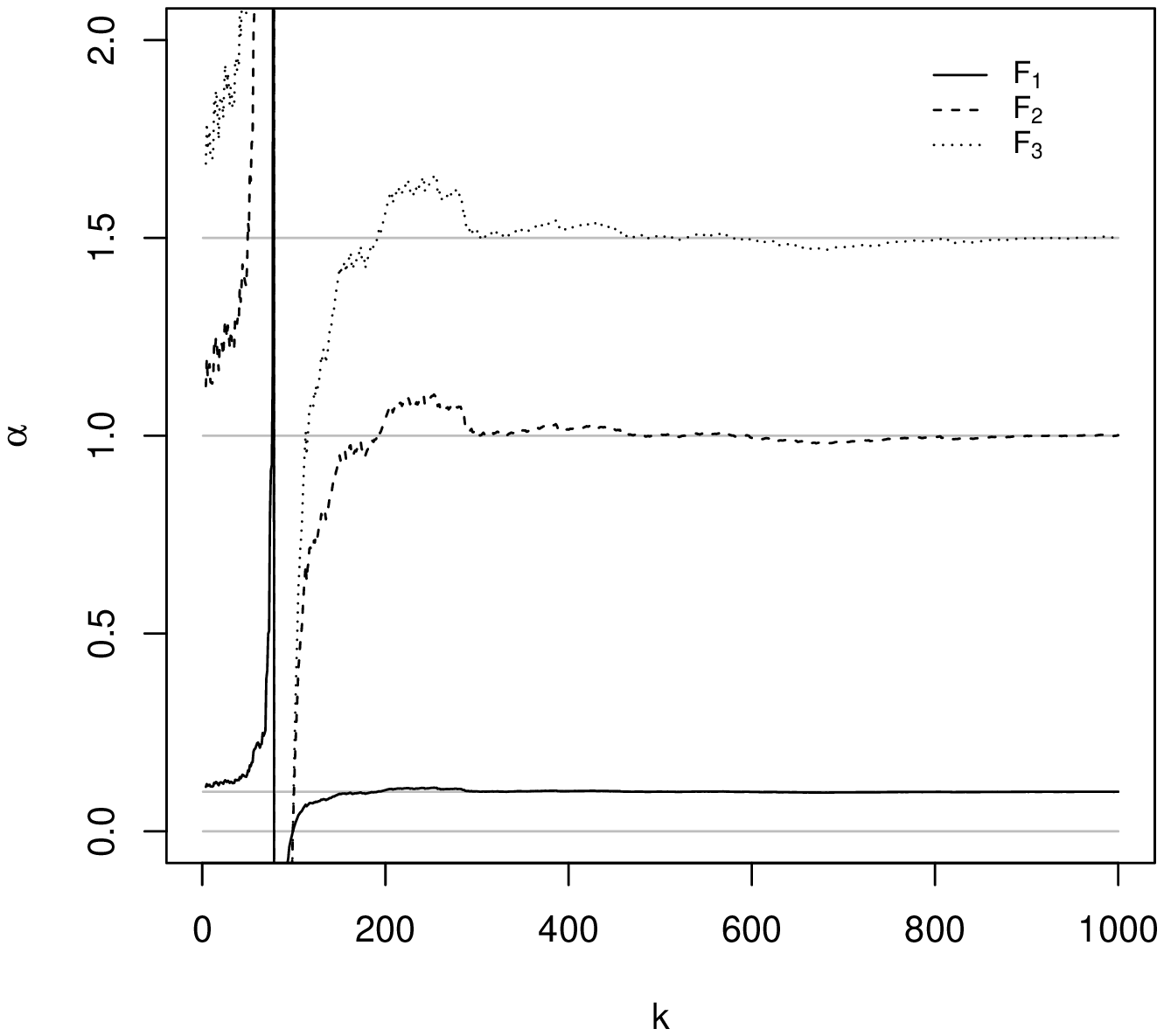}
}
\subfigure[$C$ = 0.1, $n$ = 10 000]{
\includegraphics[scale=0.29]{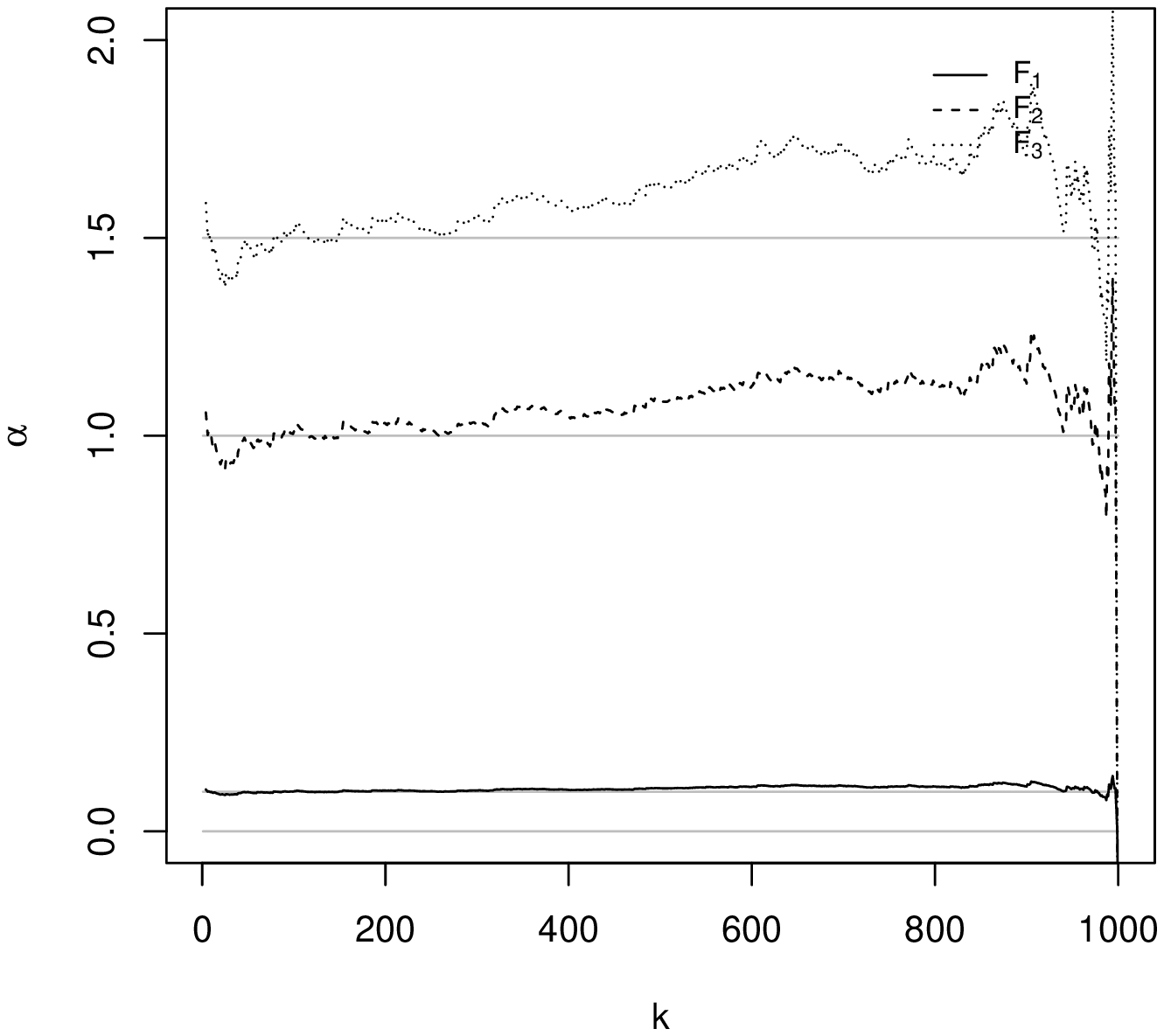}
}
\subfigure[$C$ = 0.1, $n$ = 100 000]{
\includegraphics[scale=0.29]{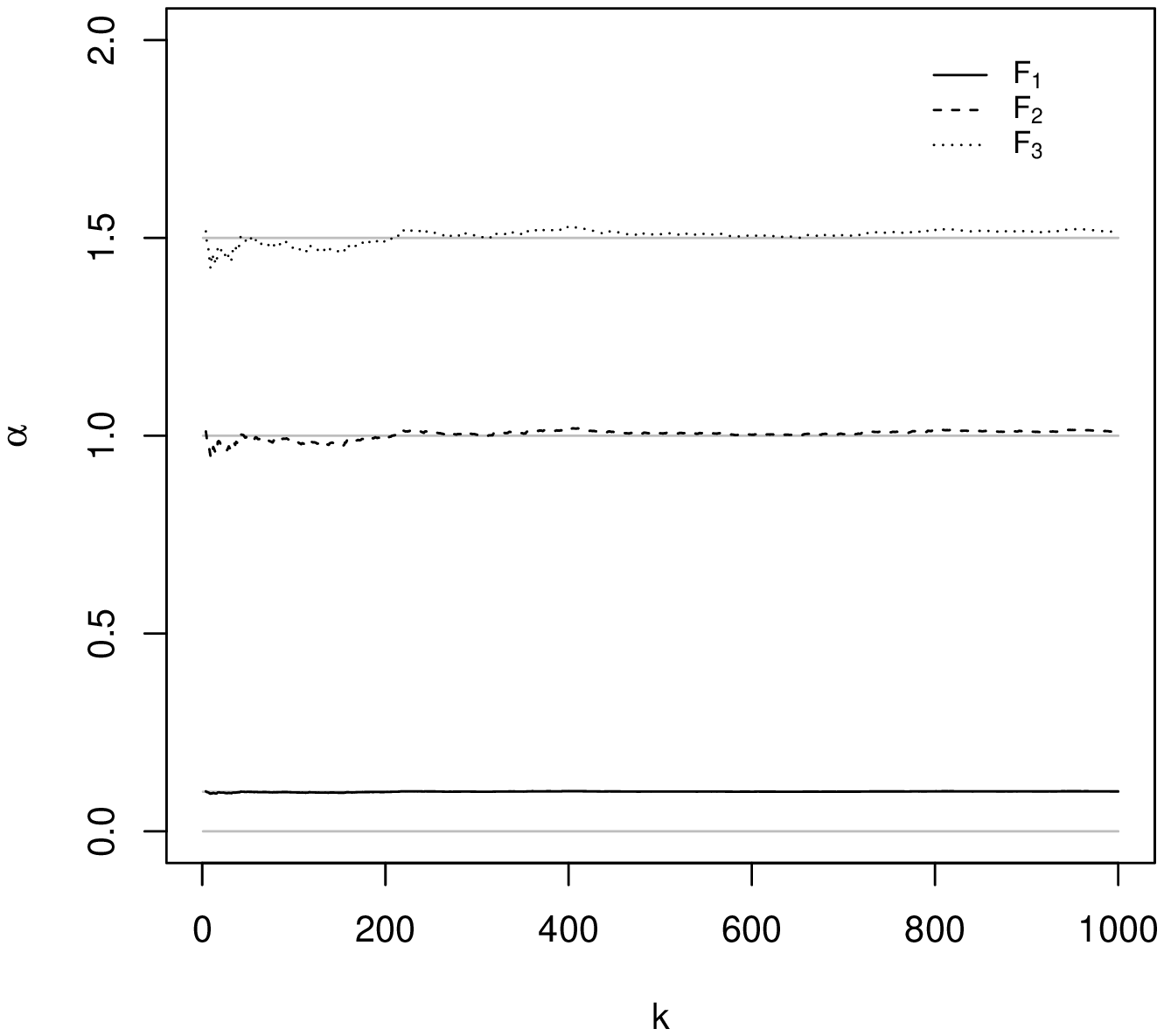}
}
\subfigure[$C$ = 1, $n$ = 1 000]{
\includegraphics[scale=0.29]{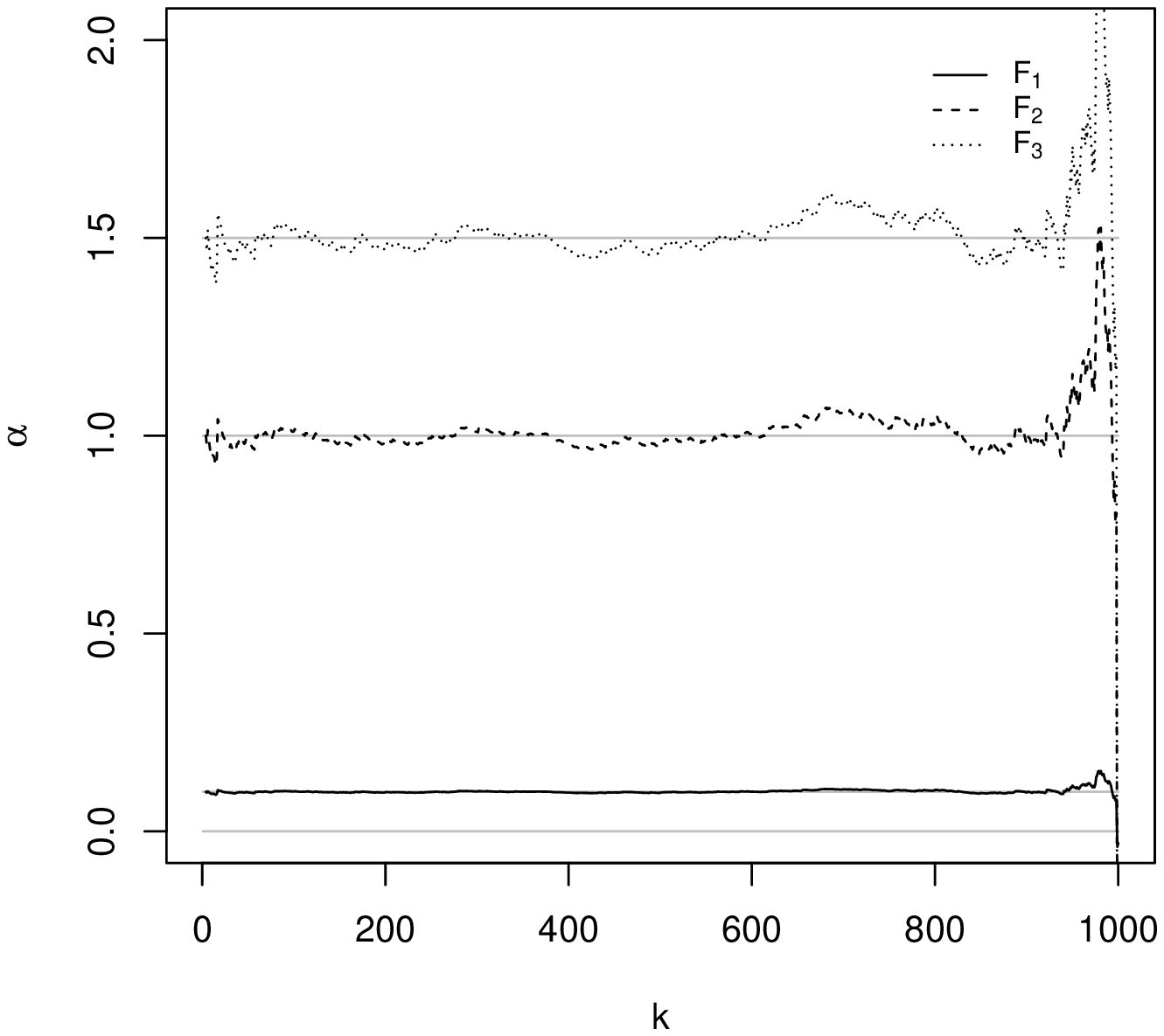}
}
\subfigure[$C$ = 1, $n$ = 10 000]{
\includegraphics[scale=0.29]{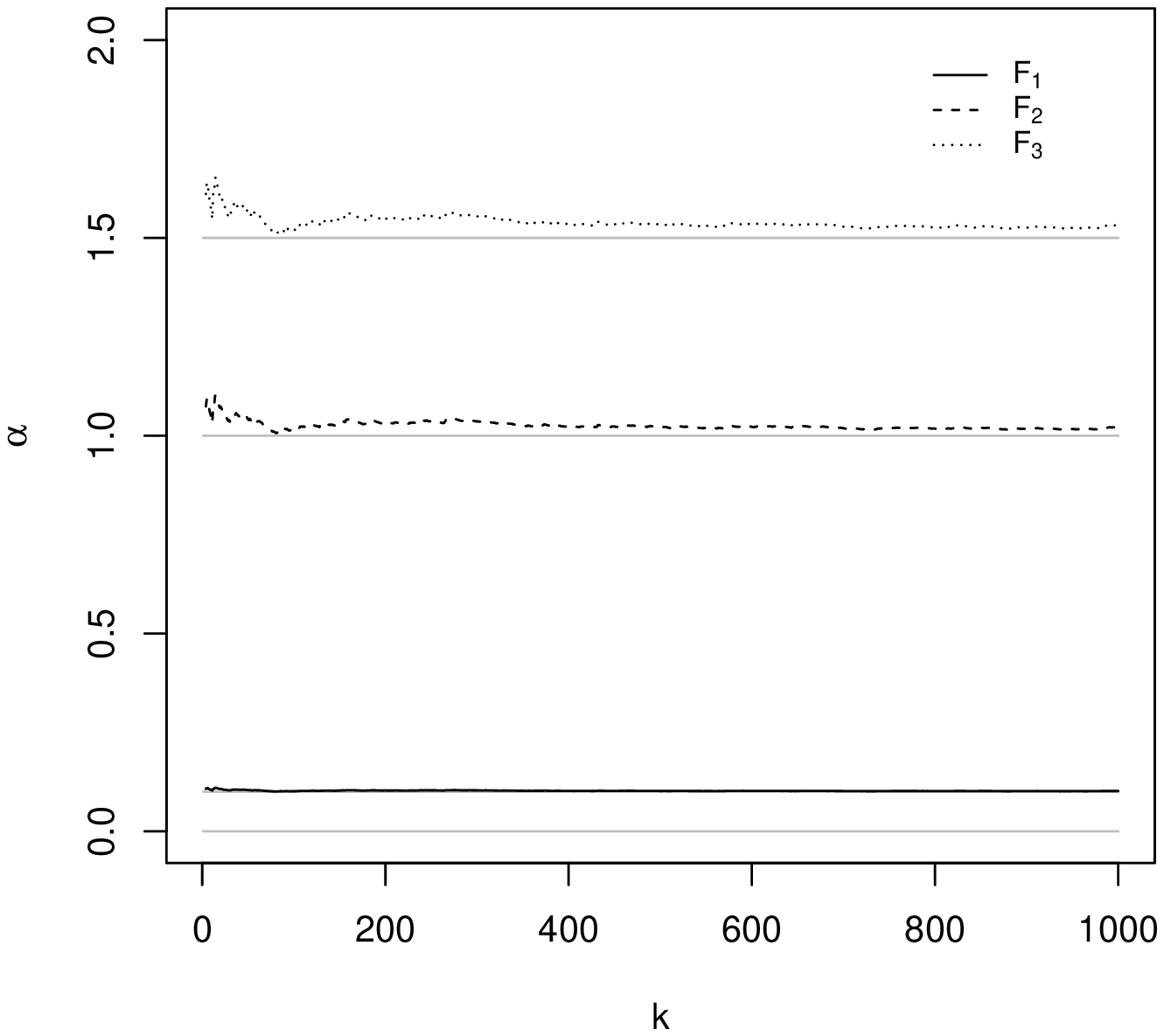}
}
\subfigure[$C$ = 1, $n$ = 100 000]{
\includegraphics[scale=0.29]{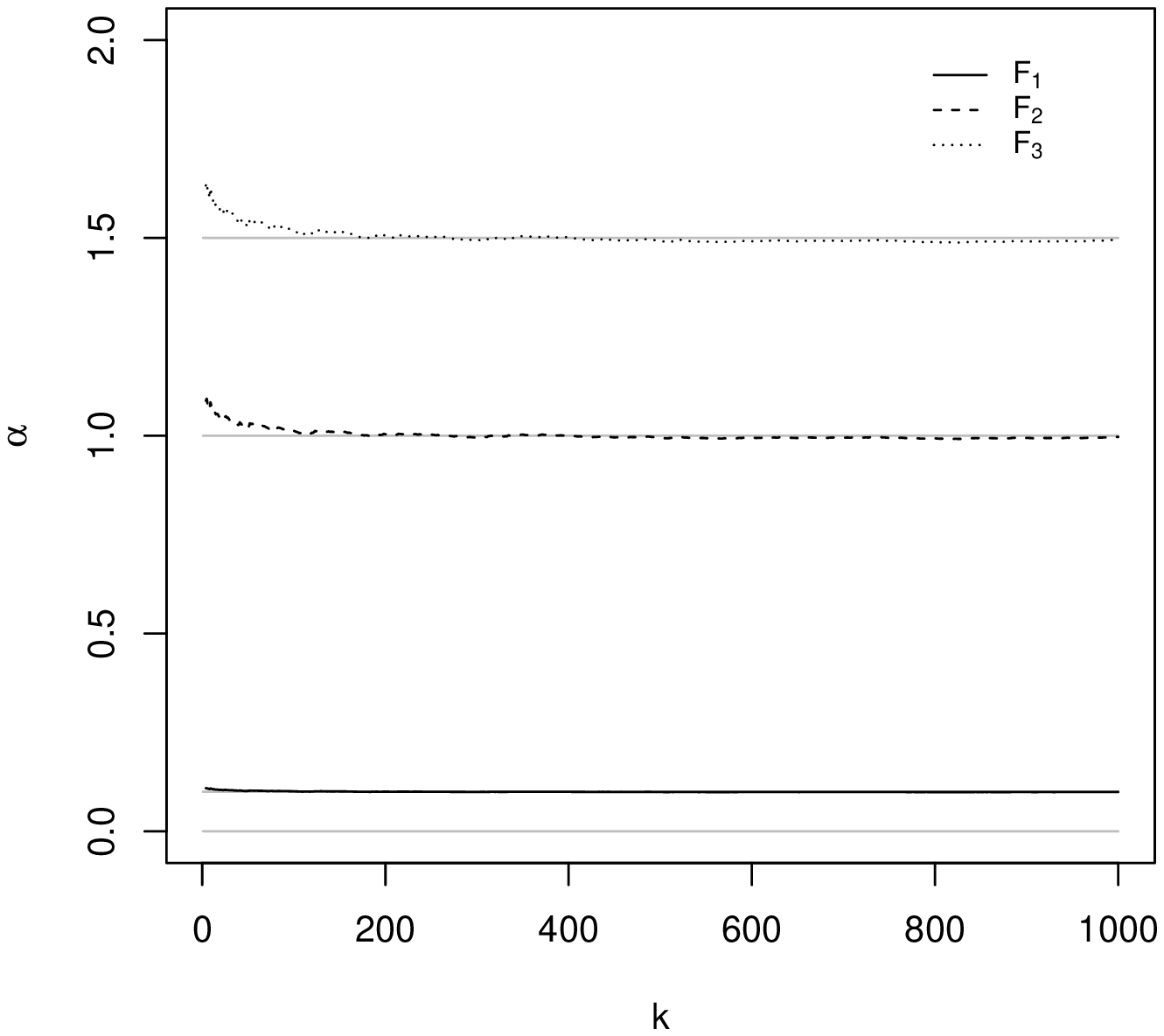}
}
\subfigure[$C$ = 10, $n$ = 1 000]{
\includegraphics[scale=0.29]{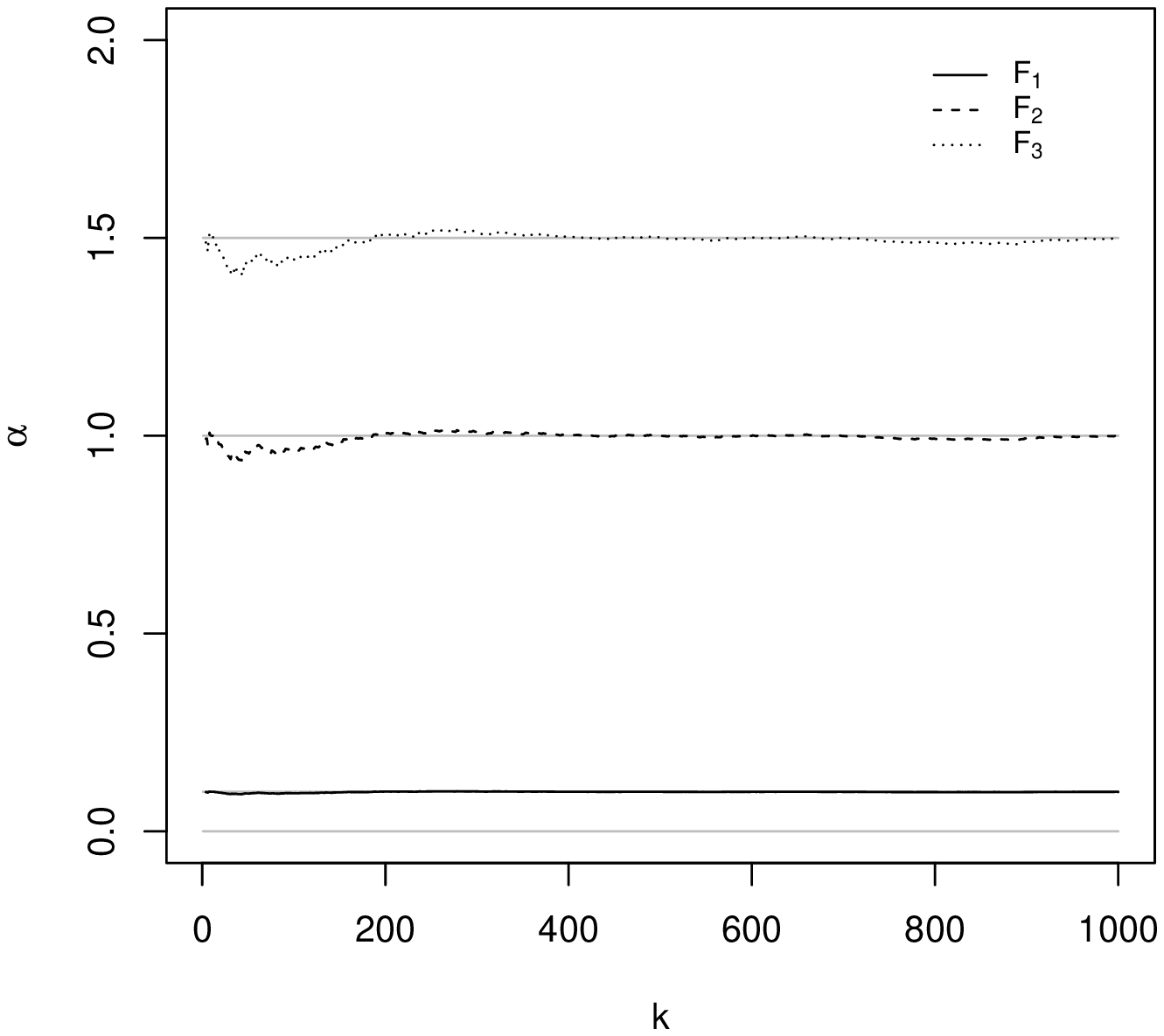}
}
\subfigure[$C$ = 10, $n$ = 10 000]{
\includegraphics[scale=0.29]{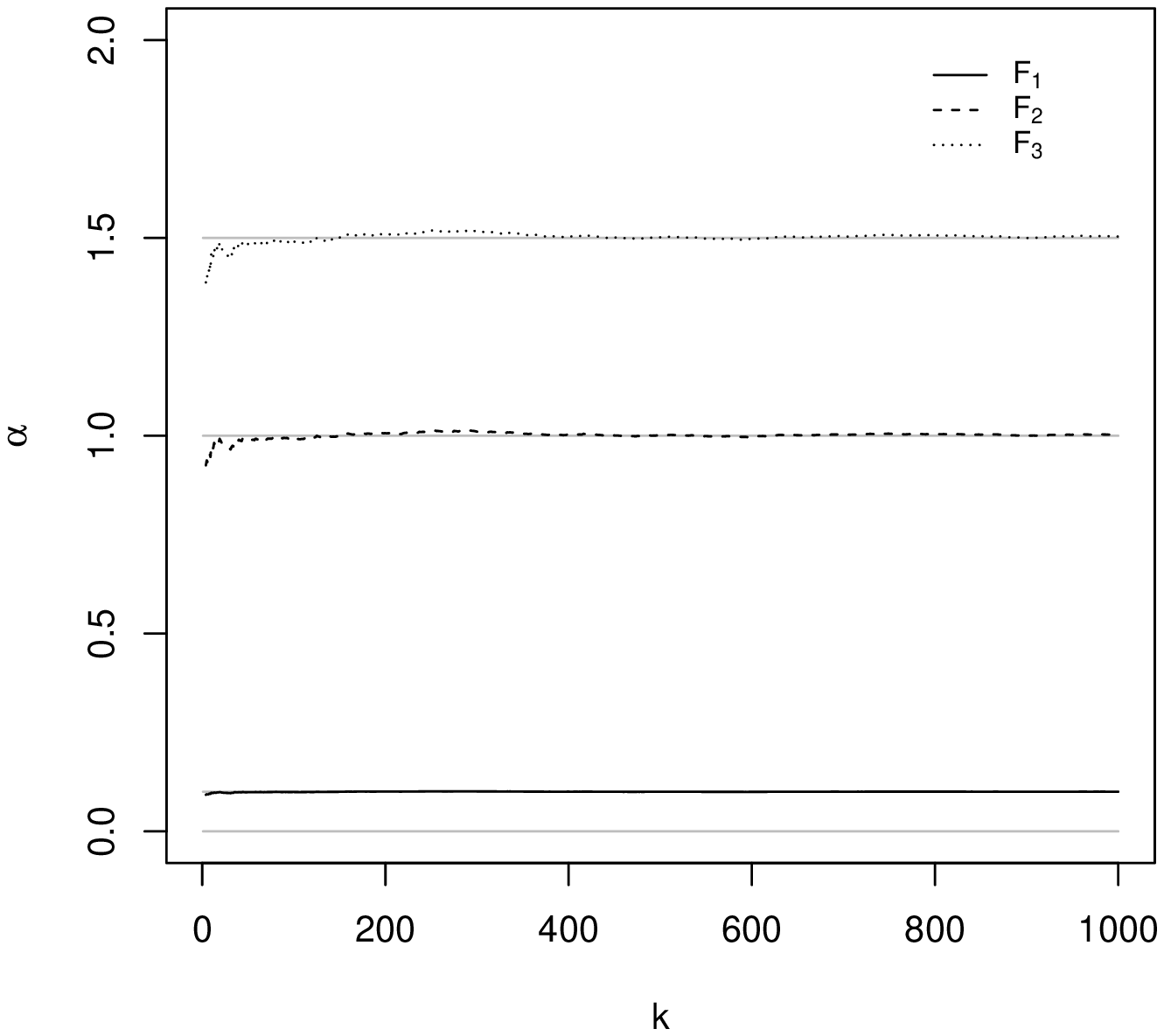}
}
\subfigure[$C$ = 10, $n$ = 100 000]{
\includegraphics[scale=0.29]{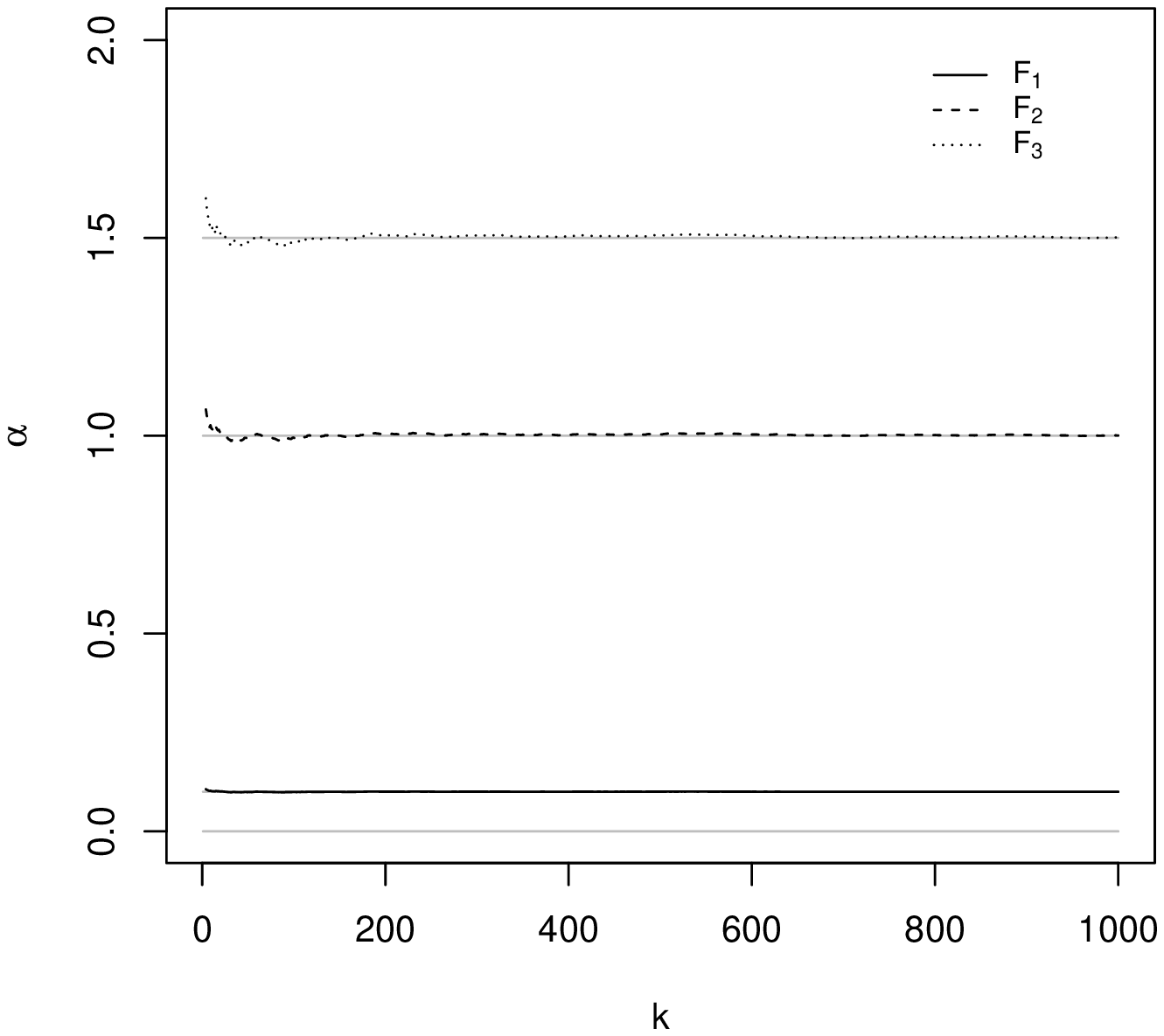}
}
\caption{Plots of $\hat{\alpha}$ against $k$, varying $C$ on row and $n$ on column}
\label{fig01d}
\end{figure}

\subsection{Real application}


In this subsection we evaluate the performance of $\hat{\alpha}$, $\hat{\alpha}_H$ and $\hat{\alpha}_M$ using Danish data on large fire insurance losses.
Data are real, they are obtained from the R SMPraticals add-on package publicly available from CRAN web page \url{{http://cran.r-project.org/}}.
This set consists of 2492 Danish fire insurance losses expressed in millions of Danish Krone (DKK) from the years 1980 to 1990 inclusive and have been adjusted to reflect 1985 values.

These data are frequently used in different contexts, namely to evaluate models or estimators.
For instance, McNeil \cite{McNeil1997} and Resnick \cite{Resnick1997} evaluate some tail estimators using extreme value theory, considering mainly losses above DKK 1 million.
The latter author concludes a likely value of $\alpha=1.45$.

We compute estimates of $\hat{\alpha}$ (fixing $C=1$), $\hat{\alpha}_H$ and $\hat{\alpha}_M$ for both all Danish data and data above DKK 1 million.
In the latter case there are 2167 observations.
Figure \ref{fig02} presents plots of these two cases.
To help the interpretation of estimates, a reference line is included at $\alpha=1.45$.

In the case of all registers, i.e. $n=2492$, all of estimators approximate to or cross $\alpha=1.45$, including the new estimator. 
However, it seems that the convergence level would be a bit lower since at least $\hat{\alpha}$ (solid line) and $\hat{\alpha}_H$ (solid line in bold) seems to oscillate around to $\alpha=1.41$.
Then, the first estimators reaching convergence seem to be $\hat{\alpha}$ and $\hat{\alpha}_H$ but presenting oscillations, when $k$ is between 250 and 500.
As $k$ is over 500 these estimators make a big oscillation that seems to be completed for $k$ over 1 000.
On the other hand, $\hat{\alpha}_M$ (dashed line in bold) shows the slowest convergence with respect to the other competitors.
In any way, it seems that this estimator converges toward $\alpha=1.41$ when $k$ increases over 1 000.

Analyzing losses above DKK 1 million, that gives a sample of $n=2167$ observations.
Then, the results obtained when $n=2492$ hold excepting for the new estimator.
This happens since $\hat{\alpha}_H$ and $\hat{\alpha}_M$ depend on the $k$ highest values observed, independently of the subsample size of the highest values of a given sample.
Unlike to these estimators, $\hat{\alpha}$ depends on the sample size.
Hence, new estimates of $\hat{\alpha}$ show a different behavior to the ones found when $n=2492$.
They show a trend to decrease from $k$ = 200, without evidencing any convergence.
This fact shows that variations of subsample sizes of the highest values of a given sample may sensibly affect the new estimator.


\begin{figure}[!ht]
\centering
\subfiguretopcapfalse
\subfigure[$n=2492$]{
\includegraphics[scale=0.45]{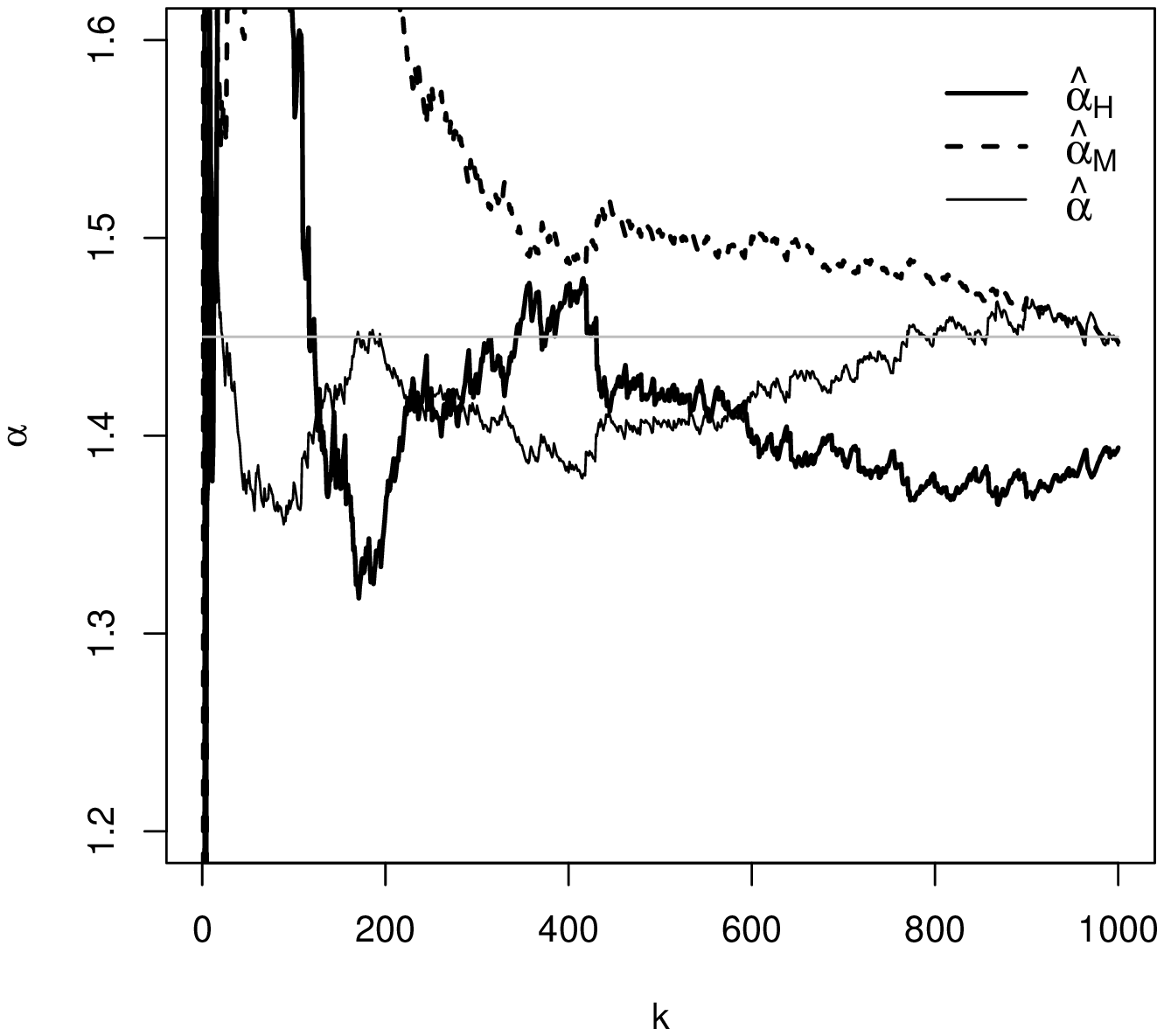}
}
\subfigure[$n=2167$]{
\includegraphics[scale=0.45]{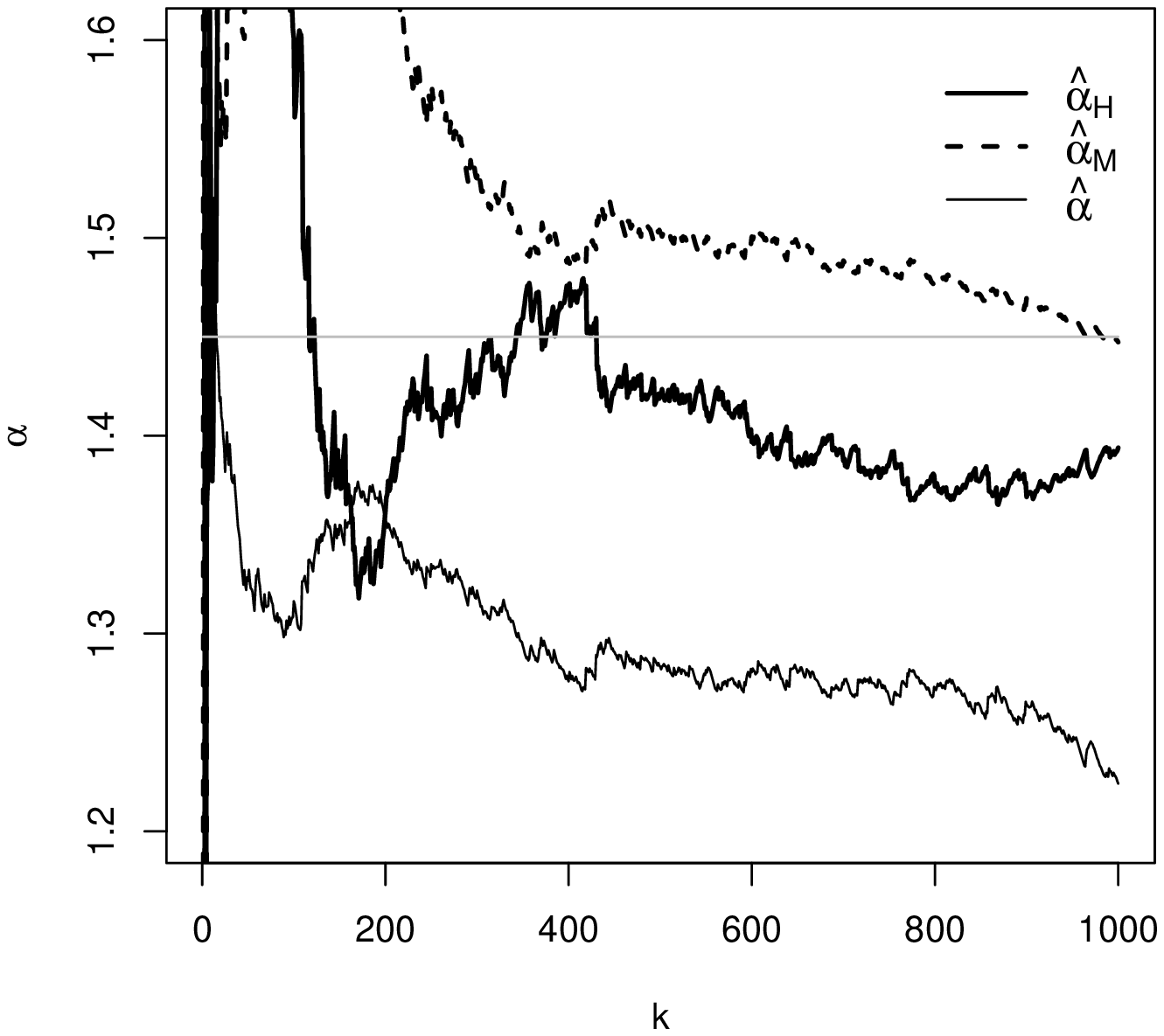}
}
\caption{Plots for some estimators of $\alpha$ against $k$, danish data of losses}
\label{fig02}
\end{figure}





\section{Conclusions}
\label{conclusions}

A simple estimator for the $\M$-index of the class $\M$, larger than the class of regularly varying functions, was formulated.
It has good properties like strong consistency and asymptotic normality, with a better rate of convergence than the one of most well-known estimators.
This last feature is illustrated in simulated examples where estimates converged faster than other competitors, excepting when scale parameters are lower than 1.
Illustrations based on simulations showed that higher scale parameters, lower shape parameters or large sample sizes improve the convergence of the new estimator.
Furthermore, in an example using real data the new estimator confirmed estimates given by some classical estimators, when the whole sample was used. 
However, estimates given by the new estimator varied when a subsample concentrating on the highest losses was used, showing the sensibility of the new estimator to variations of sample sizes.

\section*{Acknowledgments} 

The author gratefully acknowledges the support of SWISS LIFE through its ESSEC research program on 'Consequences of the population ageing on the insurances loss'.
I am indebted to Marie Kratz for her reading of an early version of this manuscript.


\appendix

\section{Proof of Theorem \ref{teoCK}}

Let $\eta\in\Rset$ and let $U:\Rset^+\to\Rset^+$ be a measurable function.

\begin{proof}[Proof of the necessary condition]
Let $\epsilon>0$ and $U\in\M$ with $\M$-index $\eta$.
One has, by definition, that
$$
\limx\frac{U(x)}{x^{\eta+\epsilon}}=0
\quad\textrm{and}\quad
\limx\frac{U(x)}{x^{\eta-\epsilon}}=\infty\textrm{.}
$$
Hence, there exists $x_0\geq1$ such that, for $x\geq x_0$,
$$
U(x)\leq\epsilon x^{\eta+\epsilon}
\quad\textrm{and}\quad
U(x)\geq\frac{1}{\epsilon}x^{\eta-\epsilon}\textrm{.}
$$
Applying the logarithm function to these inequalities and dividing them by $\log(x)$ (with $x>1$) provide
$$
\frac{\log\left(U(x)\right)}{\log(x)}\leq\frac{\log\left(\epsilon\right)}{\log(x)}+\eta+\epsilon
\quad\textrm{and}\quad
\frac{\log\left(U(x)\right)}{\log(x)}\geq-\frac{\log\left(\epsilon\right)}{\log(x)}+\eta-\epsilon\textrm{,}
$$
and, one then has
$$
\limsupx\frac{\log\left(U(x)\right)}{\log(x)}\leq\eta+\epsilon
\quad\textrm{and}\quad
\liminfx\frac{\log\left(U(x)\right)}{\log(x)}\geq\eta-\epsilon\textrm{,}
$$
from which one gets, taking $\epsilon$ arbitrary,
$$
\eta\leq\liminfx\frac{\log\left(U(x)\right)}{\log(x)}\leq\limsupx\frac{\log\left(U(x)\right)}{\log(x)}\leq\eta\textrm{,}
$$
and the assertion follows.
\end{proof}

\begin{proof}[Proof of the sufficient condition]
Let $\epsilon>0$.
By hypothesis, there exists $x_0>1$ such that, for $x\geq x_0$, $\big|\log(U(x))\big/\log(x)-\eta\big|\leq\epsilon\big/2$.

Writing, for $w\in\big\{\epsilon,-\epsilon\big\}$,
$$
\frac{U(x)}{x^{\eta+w}}
=\exp\left\{\log(x)\times\left(\frac{\log(U(x))}{\log(x)}-\eta-w\right)\right\}
$$
gives
$$
\exp\left\{\log(x)\times\left(-\frac{\epsilon}{2}-w\right)\right\}
\leq
\frac{U(x)}{x^{\eta+w}}
\leq
\exp\left\{\log(x)\times\left(\frac{\epsilon}{2}-w\right)\right\}\textrm{,}
$$
and then,
$$
\limx\frac{U(x)}{x^{\eta+\epsilon}}\leq
\limx\exp\left\{\log(x)\times\left(\frac{\epsilon}{2}-\epsilon\right)\right\}=0
$$
and
$$
\limx\frac{U(x)}{x^{\eta-\epsilon}}\geq
\limx\exp\left\{\log(x)\times\left(\frac{\epsilon}{2}+\epsilon\right)\right\}=\infty\textrm{.}
$$
These two limits provide $U\in\M$ with $\M$-index $\eta$.
\end{proof}

\end{document}